\documentclass[12pt]{amsart}
\usepackage[utf8]{inputenc}

\newlength{\tempfigdim}

\usepackage{amssymb, amsmath, amsthm}
\usepackage[margin=1in]{geometry}
\usepackage{tikz, tikz-cd}
\usepackage{graphicx}
\usepackage{enumerate}
\usepackage{comment}
\usepackage{bbold}
\usepackage{subcaption, caption}
\usepackage{pinlabel}
\usepackage{hyperref} % clickable toc, citations
\hypersetup{
    colorlinks = false, %set true if you want colored links
    linktoc = all,     %set to all if you want both sections and subsections linked
    linkcolor = darkgray,  %choose some color if you want links to stand out
}
\usepackage[all,cmtip]{xy}

\usepackage[bordercolor=gray!20,backgroundcolor=blue!10,linecolor=blue!50,textsize=footnotesize,textwidth=0.8in]{todonotes}
\usepackage{pinlabel}

\theoremstyle{definition}
\newtheorem{theorem}{{Theorem}}[section]
\newtheorem{lemma}[theorem]{{Lemma}}
\newtheorem{proposition}[theorem]{{Proposition}}
\newtheorem{definition}[theorem]{{Definition}}
\newtheorem{corollary}[theorem]{{Corollary}}
\newtheorem{example}[theorem]{{Example}}
\newtheorem{question}[theorem]{{Question}}

\newtheorem{remark}[theorem]{{Remark}}
\newtheorem{claim}[theorem]{{Claim}}

\newcommand{\st}{\ | \ }

\newcommand{\resp}{resp.\ }
\newcommand{\nntangles}{\mathcal{T}^{\wedge}_{n,n}}

\newcommand{\C}{\mathbb{C}}
\newcommand{\Q}{\mathbb{Q}}
\newcommand{\R}{\mathbb{R}}
\newcommand{\Z}{\mathbb{Z}}
\newcommand{\F}{\mathbb{F}}
\newcommand{\CA}{\mathcal{A}} % algebra (Frobenius in our case)

\newcommand{\gr}{\text{gr}}
\newcommand{\grh}{\gr_h}
\newcommand{\grq}{\gr_q}
\newcommand{\grqk}{\gr_{q-k}}
\newcommand{\grk}{\gr_k}
\newcommand{\mroto}{m_{r_0, t_0}}
\newcommand{\grroto}{\gr_{r_0, t_0}}
\newcommand{\sroto}{s_{r_0, t_0}}
\newcommand{\tot}{{\textit{tot}}}
\newcommand{\ftot}{{\textit{ftot}}}
\newcommand{\dftot}{\delta_{\ftot}}
\newcommand{\dd}{\mathbf{d}}
\newcommand{\ddgamma}{\mathbf{d}_{\D_\gamma}}
\newcommand{\h}{\mathbf{h}}
\newcommand{\hgamma}{\h_{\D_\gamma}}
\newcommand{\CF}{\mathcal{F}} % filtration

\newcommand{\Cftot}{\CC_\ftot}
\newcommand{\cone}{\mathrm{cone}}

\newcommand{\BBO}{\mathbb{O}}
\newcommand{\BBX}{\mathbb{X}}
\newcommand{\rk}{\mathrm{rk}} % band rank, if add subscript for braid index
\newcommand{\bclosure}{\widehat{\beta}}
\newcommand{\idclosure}{\widehat{\mathbb{1}}}
\newcommand{\sigmahat}{\widehat{\sigma}}
\newcommand{\srt}{s_{r, t}}
\newcommand{\grrt}{gr_{r,t}}
\newcommand{\vminus}{\mathbf{v}_-}
\newcommand{\Mt}{\mathfrak{M}_{t_0}}

\newcommand{\catLink}{\textbf{Link}} % links and cobords
\newcommand{\catDiag}{\textbf{Diag}} % diagrams and movies
\newcommand{\catFilt}{\textbf{Filt}} % filtered complexes up to equiv
\newcommand{\boolcube}{\mathbf{2}} % Boolean cube (the 1-skeleton of the above)
\newcommand{\Kom}{\textit{Kom}}
\newcommand{\Vect}{\textit{Vect}}

\newcommand{\CC}{\mathcal{C}} % a chain complex
\newcommand{\D}{\mathcal{D}} % a diagram; \CD is not allowed?
\newcommand{\CK}{\mathcal{K}} % a Kh-Fl theory
\newcommand{\Kh}{\mathit{Kh}} % Khovanov (co)homology
\newcommand{\Kc}{\mathit{Kc}} % Khovanov (co)chain complex
\newcommand{\Aux}{\mathit{Aux}} % auxiliary data
\newcommand{\SSS}{Sarkar-Seed-Szab\'{o}}
\newcommand{\Szabo}{Szab\'{o}}
\newcommand{\Ozsvath}{Ozsv\'{a}th}

\title[Annular link invariants from Sarkar-Seed-Szab\'{o} spectral sequence]{Annular link invariants from the Sarkar-Seed-Szab\'{o} spectral sequence}

\author[L. Truong]{Linh Truong}
\thanks{The first author was partially supported by NSF grant DMS-1606451.}
\address {School of Mathematics, Institute for Advanced Study, Princeton, NJ 08540}
\email{ltruong@math.ias.edu}

\author[M. Zhang]{Melissa Zhang}
\address {Department of Mathematics, University of Georgia, Athens, GA 30602}
\email{melissa.zhang@uga.edu}

\begin{document}
\begin{abstract}
For a link in a thickened annulus $A \times I$, we define a $\Z \oplus \Z \oplus \Z$ filtration on Sarkar-Seed-Szab\'o's perturbation of the geometric spectral sequence. The filtered chain homotopy type is an invariant of the isotopy class of the annular link. From this, we define a two-dimensional family of annular link invariants and study their behavior under cobordisms. In the case of annular links obtained from braid closures, we obtain a necessary condition for braid quasipositivity and a sufficient condition for right-veeringness, as well as Bennequin-type inequalities. 
\end{abstract}
\maketitle

\section{Introduction}

In 2000, Khovanov introduced his categorified knot invariant now known as $\mathfrak{sl}(2)$ homology or Khovanov homology \cite{Khovanov-00}. Soon after, \Ozsvath-\Szabo\ \cite{OzSz-HFK} and Rasmussen \cite{Rasmussen-HFK} independently introduced knot Floer homology, which arises from a filtration on Heegaard Floer homology. Among their many topological applications, knot Floer homology and Khovanov homology have led to a wealth of concordance invariants. 

\Ozsvath\ and \Szabo\ extracted information on the four-ball genus of a knot by constructing the concordance invariant $\tau$ \cite{OzSz-tau}, defined using a filtration induced by a knot on the Heegaard Floer chain complex of the three-sphere. Through a similar algebraic construction as that for $\tau$, Rasmussen defined a concordance invariant $s$ \cite{Rasmussen-s}, which is derived from a filtration spectral sequence from Khovanov homology to Lee's perturbation \cite{Lee-perturbation}  of Khovanov homology.

\iffalse
The general algebraic structure of these invariants has been used to define generalizations of these invariants in many different settings. One begins with a $\Z^r$ filtered chain complex $\CC$ with a finite set of distinguished, homogeneous generators. In other words, there are $r$ gradings on a finite set of generators, each of which gives rise to a filtration on the complex. Any positive linear combination $\alpha$ of the $r$ gradings thus also gives a filtration $\CF_\alpha$ on the total complex. 
One also has a distinguished homology class $g$ in the total homology. If the chain homotopy equivalence class of $\CC$ was an invariant of the object of study, then filtration level of $g$ in homology (defined appropriately) is a numerical invariant. 
\fi

The concordance invariants $\tau$ and $s$ are defined using a single filtration grading on a Heegaard Floer or Khovanov complex. Concordance invariants constructed using two linearly independent filtration gradings include
\begin{itemize}
    \item \Ozsvath-Stipsicz-\Szabo's concordance homomorphism $\Upsilon_t$ \cite{OSS-Upsilon}, a piecewise linear function that generalizes the  \Ozsvath-\Szabo\ $\tau$ invariant. The $\Upsilon_t$ invariant is defined using the $\Z \oplus \Z$-filtered knot Floer chain complex $\text{CFK}^\infty(K)$. 
    \item Grigsby-Licata-Wehrli's $d_t$ \cite{GLW-dt}, a smooth annular link invariant with applications to annular link cobordisms and braids, defined from the $\Z \oplus \Z$-filtered annular Khovanov-Lee complex.
    \item the \SSS\ generalized Rasmussen invariants $s^{\mathcal{U}}$ \cite{SSS} derived from a perturbation of the Szab\'o geometric spectral sequence and the Bar Natan complex. 
    \item Lewark-Lobb's smooth concordance invariants $\gimel_n$ \cite{LL-gimel}, piecewise linear functions arising from Khovanov-Rozansky's complex for $\mathfrak{sl}_n$ knot cohomology over the ring $\C[x]/x^{n-1}(x-1)$.
    \item Dai-Hom-Stoffregen-Truong's  concordance homomorphisms $\phi_j$ \cite{DHST} defined using the knot Floer chain complex over the ring $\mathcal{R} = \F[U,V]/UV$.
\end{itemize}
We study the \SSS\ perturbation of the geometric spectral sequence in the setting of annular links, that is, isotopy classes of links in a thickened annulus $A \times I$, or equivalently, a solid torus $D^2 \times S^1$.
We will show that for an annular link, the \SSS\ complex admits an annular filtration. This extra filtration allows us to define a two-dimensional family of real-valued annular link invariants $\srt$, by using three filtration gradings on the \SSS\ complex. This family of invariants recovers the \SSS\ generalized $s$-invariants and shares many properties with Grigsby-Licata-Wehrli's $d_t$ annular link invariants from Khovanov-Lee homology (cf.\ \cite{GLW-dt}, Theorem 1). 

\begin{theorem}
\label{thm:mainINTRO}
Let $(L, o)$ be an annular link $L \subset A \times I$ equipped with an orientation $o$. Let $r \in [0,1]$, $ t \in [0,1]$.
\begin{enumerate}
    \item
    For each pair $r,t \in [0,1]$, $s_{r,t}(L,o)$ is an oriented annular link invariant.
    \item Suppose $L$ is a knot. Then $s_{0,0}(L,o) =  s_{\F_2}(L,o)-1$, where $s_{\F_2}$ is Rasmussen's concordance invariant over $\F_2$ coefficients \cite{turner-calculating, mackaay-turner-vaz}.
    \item 
    For fixed $r$ (respectively  $t$), the function $s_{r,t}(L,o)$ is piecewise-linear with respect to the variable $t$ (respectively $r$). 
    \item 
    Let $\omega$ be the wrapping number of $L$. Fix $r \in [0,1]$. Then for all $t_0 \in [0,1]$
    \[
        \left ( -\frac{1}{1-r} \right )
        \lim_{t \to t_0+}
        \frac{
            \srt(L,o) 
            - s_{r,t_0}(L,o)
        }{t - t_0}
        \in \{ -\omega, -\omega + 2, \ldots, \omega-2, \omega\}.
    \]
    %Let $\omega$ be the wrapping number of $L$. Then \[ \left ( - \frac{1}{1-r} \right ) \frac{\partial s_{r,t}(L,o)}{\partial t} \in \{ -\omega, -\omega+2, \ldots, \omega -2, \omega\}\]
    %where $\partial / \partial t$ is the right derivative.
    \item 
    Let $F: (L, o) \to (L', o')$ be an oriented cobordism between two nonempty, oriented links such that each component of $F$ has a boundary component in $L$. Let $a_0$ be the number of annular births or deaths, $a_1$ the number of saddles, and $b_0$ the number of non-annular births or deaths. Then
    \[
        s_{r,t}(L,o) - s_{r,t} (L', o') \leq (r-1)(a_0 - a_1 + b_0(1-t)).
    \]
    If furthermore each component in $F$ has a boundary component in $L'$ as well, then 
    \[
        \vert s_{r,t}(L,o) - s_{r,t} (L', o') \vert \leq (r-1)(a_0 - a_1 + b_0(1-t)).
    \]
    \item $\srt(L,o)$ is an annular concordance invariant.
\end{enumerate}
\end{theorem}

We apply properties of the annular link invariants in Theorem \ref{thm:mainINTRO} to the study of braid closures. Let $\sigma \in \mathfrak{B}_n$ be an $n$-braid. Its braid closure $\sigmahat$ is naturally an annular link, equipped with a braid-like orientation $o_\uparrow$ as described in Section \ref{subsection:braids}. We abbreviate $\srt(\sigmahat, o_\uparrow)$ by $\srt(\sigmahat)$. Analogous to properties of the $d_t$ invariant \cite[Theorem 2, Theorem 4]{GLW-dt}, we obtain a \emph{necessary} condition for braid quasipositivity and a \emph{sufficient} condition for right-veeringness. For a fixed $r_0 \in [0,1]$, we find the function $s_{r_0, t}(\sigmahat)$ is piecewise-linear and has slope bounded above by $n$, the braid index of $\sigma$. Let $\mroto(L,o)$ denote the right-hand slope, with respect to the variable $t$, of $s_{r_0, t}(L,o)$ at $t_0$.  

\begin{theorem}
If $\sigma$ is a quasipositive braid of index $n$ and writhe $w \geq 0$, we have
\[\srt(\sigmahat) = (1-r)(w - (1-t)n) \]
for all $r \in [0,1]$ and $t \in [0,1]$.
\end{theorem}

\begin{theorem}
Let $\sigma \in \mathfrak{B}_n$. Fix $ r_0 < 1$. If $s_{r_0, t}(\sigmahat)$ attains maximal slope at some $t = t_0 < \frac{1}{2}$ (that is, $\mroto(\sigmahat) = n$ for some $t_0 \in [0, \frac{1}{2})$), then $\sigma$ is right-veering. 
\end{theorem}

In addition, we study the behavior of $\srt(\sigmahat)$ under Markov stabilizations. We achieve bounds which can be compared to the analogous $d_t$ bounds (\cite{GLW-dt}, Proposition 5). 
\begin{proposition}
Let $\sigma \in \mathfrak{B}_n$ and suppose $\sigma^\pm \in \mathfrak{B}_{n+1}$ is obtained from $\sigma$ by either a positive or negative Markov stabilization. Then for all $r \in [0,1]$, $t \in [0,1]$,
\[ \srt(\sigmahat) - (1-r)t  \leq \srt(\sigmahat^\pm)  \leq  \srt(\sigmahat) + (1-r)t . \]
\end{proposition}

A complete understanding of the behavior of the $\srt$ (and similarly, $d_t$) annular link invariants under positive and negative stabilization is currently unknown. This question is related to the question of the effectiveness of transverse invariants obtained from Khovanov homology. 

\begin{question} \label{ques:stabilization}
Suppose $\sigma^+ \in \mathfrak{B}_{n+1}$ is obtained from $\sigma \in \mathfrak{B}_n$ by a positive Markov stabilization. Does the equality 
\[  \srt(\sigmahat^+)  =  \srt(\sigmahat) + (1-r)t  \]
hold for all $r \in [0,1]$, $t \in [0,1]$?
\end{question}
\noindent
A positive answer to Question \ref{ques:stabilization} would yield a new transverse invariant: $\srt(\sigmahat) - n(1-r)t $. It is unclear whether such an invariant would be \emph{effective}, in the sense that it would provide more information than the self-linking number of the transverse link. The analogous questions with respect to the $d_t$ invariant are open, as well as the effectiveness of Plamenevskaya's transverse link invariant $\psi$ \cite{Plamenevskaya-transverse}. 

Analogous to the annular Khovanov-Lee setting \cite[Corollary 4]{GLW-dt}, we produce a lower bound on the \emph{band rank} $\rk_n$ of a braid \cite{Rudolph-band-rank}. Given $\beta \in \mathfrak{B}_n$, 
    \[
        \rk_n(\beta):= \min \left \{ c \in \Z^{\geq 0} \ \bigg \vert \  \beta = \prod_{j=1}^c \omega_j \sigma_{i_j}^\pm (\omega_j)^{-1} \text{ for some } \omega_j \in \mathfrak{B}_n \right \},
    \]
    where $\sigma_{i_j}$ denotes the elementary Artin generators. Note that band rank $\rk_n$ is a braid conjugacy class invariant. 
    Topologically, band rank $\rk_n(\beta)$ is the minimum number of half-twist bands (running perpendicularly to the strands) needed to construct a Seifert surface for $\bclosure$ from $n$ disks (the obvious Seifert surface for the identity braid closure in $\mathfrak{B}_n$). 

\begin{proposition}
Let $r \neq 1$. 
Given an oriented cobordism $F$ from $(L,o)$ to $(L',o')$ with $a_0$ annular even index critical points, $a_1$ annular odd index critical points, and $b_0$ non-annular even index critical points, we have
\[
 \left | \frac{s_{r,t}(\hat\beta)}{1-r} + n(1-t) \right | \leq rk_n(\beta).
 \]
\end{proposition}
As described in \cite{GLW-dt, Grigsby-ribbon}, work of Rudolph \cite[Sec. 3]{Rudolph-band-rank} suggests that a braid conjugacy class invariant (like $\srt$) that yields a lower bound on band rank could potentially lead to an effective ribbon obstruction. This question is explored using the $d_t$ invariant in \cite{Grigsby-ribbon}; we do not pursue this question using $\srt$ here.

The \SSS\ theory conjecturally relates the Khovanov and Heegaard Floer worlds by combining the Szab\'o geometric spectral sequence with Bar Natan's homology. Szab\'o's geometric spectral sequence \cite{Szabo-gss} is conjecturally isomorphic to the Ozsv\'ath-Szab\'o spectral sequence \cite{OzSz-KhHF} relating the Khovanov homology of (the mirror of) a link $\overline{L}$ to the Heegaard Floer homology $\widehat{\textit{HF}}(\Sigma(L))$ of the double branched cover. In addition, F. Lin \cite{Lin} recently constructed a spectral sequence relating a version of Bar Natan's homology of (the mirror of) a link $\overline{L}$ to the involutive monopole Floer homology $\widetilde{\textit{HMI}}(\Sigma(L))$ of the double branched cover $\Sigma(L)$, which is analogous to the hat version of the involutive Heegaard Floer homology $\widehat{\textit{HFI}}(\Sigma(L))$ of Hendricks-Manolescu \cite{Hendricks-Manolescu}. 

% The \SSS\ theory conjecturally relates the Khovanov and Heegaard Floer worlds by combining the Szab\'o geometric spectral sequence with Bar Natan's homology. Szab\'o's geometric spectral sequence \cite{Szabo-gss} is conjecturally isomorphic to the Ozsv\'ath-Szab\'o spectral sequence \cite{OzSz-KhHF} relating Khovanov homology of (the mirror of) a link $\overline{L}$ to the Heegaard Floer homology $\widehat{\textit{HF}}(\Sigma(L))$ of the double branched cover. In addition, Francesco Lin \cite{Lin} recently constructed a spectral sequence from Bar Natan's homology of (the mirror of) a link $\overline{L}$ to the involutive monopole Floer homology $\widetilde{\textit{HMI}}(\Sigma(L))$ of the double branched cover $\Sigma(L)$, which is analogous to the hat version of the involutive Heegaard Floer homology $\widehat{\textit{HFI}}(\Sigma(L))$ of Hendricks-Manolescu \cite{Hendricks-Manolescu}. 

We will study a version of the \SSS\ theory for annular links. In \cite{Roberts-dbc} Roberts constructed an annular analogue to the Ozsv\'ath-Szab\'o spectral sequence \cite{OzSz-KhHF}. Therefore, the annular version of Szab\'o's geometric spectral sequence conjecturally converges to (a variant of) the knot Floer homology of the preimage $\widetilde{B}$ of the annular axis $B$ in the double branched cover.
While we study a filtered version of the \SSS\ complex, the full \SSS\ complex is defined over the two-variable polynomial ring $\F_2[H,W]$. It may also be interesting to extend the techniques in this work to, for instance, an annular version of the \SSS\ complex over $\F_2[H]$ and compare the resulting spectral sequence with that of F. Lin \cite{Lin}, potentially yielding a spectral sequence abutting to an involutive knot Floer invariant of $\widetilde{B}$ in the double branched cover $\Sigma(L)$.

% We will study a version of the \SSS\ theory for annular links. In light of the above notions, the annular version of Szab\'o's geometric spectral sequence conjecturally converges to the knot Floer homology of the preimage $\widetilde{B}$ of the annular axis $B$ in the double branched cover $\widehat{\textit{HFK}}(\widetilde{B}, \Sigma(L))$. 

\subsection{Acknowledgements} We thank Eli Grigsby for her advice and insightful comments on a draft. We also thank John Baldwin, Ciprian Manolescu, Gage Martin, Krzysztof Putyra, Adam Saltz, Sucharit Sarkar, and Matt Stoffregen for interesting conversations. This collaboration began at the Georgia International Topology Conference at the University of Georgia in 2017, and continued at Columbia University while the first author was a postdoc and the second author was a visiting PhD student. We thank UGA and Columbia's math departments for their hospitality. 

\section{Background}
\label{sec:background}

The purpose of this section is to review the homology theories used in this paper and to set notation.

\subsection{Links and their diagrams}
We are concerned with defining a functorial invariant for links in $S^3$ via link diagrams. Following \cite{BHL-funct}, we first define the relevant categories.

\begin{definition}
Let $\catLink$ denote the category of \emph{smooth} links in $S^3 = \R^3 \cup \{\infty\}$.
Objects are smooth isotopy classes of oriented links in $S^3$.
Morphisms are smooth isotopy classes of oriented, collared link cobordisms in $S^3 \times [0,1]$: two surfaces from $L_0 \subset S^3 \times \{0\}$ to $L_1 \subset S^3 \times \{1\}$ are equivalent if there is a smooth isotopy taking $F$ to $F'$ that fixes a collar neighborhood of the boundary $S^3 \times \{0,1\}$, pointwise.
\end{definition} 

Link invariants are typically computed from a link diagram.

\begin{definition}
Let $\catDiag$ denote the diagrammatic link cobordism category, representing links in $S^3$. 
Objects are oriented link diagrams in $S^2$, up to planar isotopy and the Reidemeister moves (see Remark \ref{rmk:reid-s2-r2}).
Morphisms are diagrammatic cobordisms, usually described as movies, up to a set of equivalences.

Movies are scans of surfaces in $S^3 \times [0,1]$, where $[0,1]$ is the time component. The scan at a fixed time is called a \emph{frame}, which can either look like a link diagram or something more singular. We can always arrange the link cobordism via isotopy such that there are only finitely many, isolated singular frames. In this situation, the movie is a composition of planar isotopies interspersed with finitely many \emph{elementary string interactions (ESIs)}: Reidemeister moves or oriented diagrammatic 0-, 1-, and 2-handle attachments. We sometimes refer to the handle attachments as births, saddles, and deaths, respectively, because a tiny snippet of the movie surrounding these critical levels looks like the birth or death of a circle, or describes a saddle cobordism. We visually depict ESIs by drawing the first and last frames of this tiny snippet.

Two movies are equivalent if they can be related by a finite sequence of the following moves (\cite{GLW-schur-weyl}, Appendix):
	\begin{itemize}
	\item  a Carter-Saito movie move \cite{Carter-Saito}, localized to a disk in $\R^2$,
	\item a time-level preserving isotopy of the associated immersed surfaces,
	\item an interchange of the time-levels of distant (noninteracting) ESIs.
	\end{itemize}
\end{definition}

\begin{remark}
\label{rmk:reid-s2-r2}
Since every link diagram will miss some point on $S^2$, we actually draw the diagrams on $\R^2$ and think of isotopies in $\R^2$. By Reidemeister's theorem, moving a strand across $\infty$ is equivalent to a finite sequence of Reidemeister moves through diagrams on $\R^2$.
\end{remark}

A functorial link invariant is a functor from $\catLink$. Baldwin-Hedden-Lobb describe how to lift a functor from $\catDiag$ to a functor from $\catLink$, and show that these two categories are equivalent \cite{BHL-funct}. 
In this paper we will focus on functors into a category of filtered complexes.

\begin{definition}
\label{def:cat-filt}
Let $\catFilt$ be the homotopy category of filtered complexes. The objects are filtered chain complexes, up to chain homotopy equivalence (we denote this relation by $\simeq$), and the morphisms are induced by filtered chain morphisms.
\end{definition}

\subsection{Annular links and cobordisms}
\label{subsec:annular-cobs}

Let $A$ be a closed, oriented annulus and let $I = [0,1]$ be the closed, oriented unit interval.  Via the identification 
\[A \times I = \{(r,\theta, z) \ | \ r \in [1,2], \theta \in [0, 2\pi], z \in [0,1]\} \subset (S^3 = \R^3\cup \infty),\]
any link $L \subset A \times I$ may be naturally viewed as a link in the complement of a standardly embedded unknot, $U = (z\text{-axis} \cup \infty) \subset S^3$. Such an \emph{annular link} $L \subset A \times I$ admits a diagram $\D(L) \subset A$, obtained by projecting a generic isotopy class representative of $L$ onto $A \times \{1/2\}$.

We shall view $\D(L)$  as a diagram on $S^2 \setminus \{\BBO, \BBX\}$ where $\BBX$ (\resp $\BBO$) are basepoints on $S^2$ corresponding to the inner (resp. outer) boundary circles of $A$. Note that if we forget the data of $\BBX$, we may view $\D(L)$ as a diagram on $\R^2 =  S^2 \setminus \BBO$ of $L$, viewed as a link in  $S^3$.  

In the annular context, we can categorize  these moves as annular or non-annular, based on whether they interact with the axis:
\begin{itemize}
    \item (1-handles) By transversality, all saddle moves can be thought of as annular.
    \item (0, 2-handles)  There are annular and non-annular births/deaths, depending on whether the circles being born/dying is trivial or not, respectively.
    \item (Reidemeister moves and planar isotopies) If the local disk in which a Reidemeister move occurs does not include the basepoint, then it is called an annular Reidemeister move; otherwise, it is non-annular. The same goes for planar isotopies.
\end{itemize}

\begin{remark} 
The homology theories / TQFTs of interest will still be defined as functors $\catLink \to \catFilt$; in particular, the distinguished homology classes come from the link diagram on $S^2$. However, the basepoints indicating the presence of the embedded unknot will provide an extra filtration. 
\end{remark}

\subsection{Khovanov (co)chains, gradings, and annular filtration}
\label{subsec:Khovanov-homology}

In this subsection we give a brief review of Khovanov homology and the variants appearing in this paper. We refer the reader to \cite{BN-Kh-intro, Khovanov-00} for more details on the construction, and focus mainly on setting the notation to be used in the rest of the paper.

Let $\boolcube^n$ be the $n$-dimensional Boolean cube $\{0,1\}^n$. The vertices $u \in \boolcube^n$ are graded by $|u| = \sum_{i} u_i$, and there is a directed edge $u \to v$ whenever $u \preceq v$ and $|v|-|u|= 1$.

Let $\D$ be an oriented link diagram with $n$ crossings. The Khovanov (co)chain complex $\Kc(\D)$ lies above $\boolcube^n$. In this section we'll only describe the (co)chains, and leave the differentials to a more general discussion in Section \ref{subsec:sss-defn}. 

Each vertex $u \in \boolcube^n$ corresponds to a complete resolution $\D_u$ of $\D$. Let $S_u$ be the collection of planar circles in this resolution. 
A \emph{marked resolution} is this resolution $\D_u$ where each circle is decorated with either a $+$ or a $-$. 
The vector space $\Kc$ associates to $u$ is generated by all these marked resolutions. 
Equivalently, let $\mathbb{V}$ be the two-dimensional vector space generated by $x_-$ and $x_+$; $\Kc(\D_u)$ is then $\mathbb{V}^{\otimes |S_u|}$, where the marked resolution corresponds to a pure tensor.
Equivalently, the marked resolution corresponds to a monomial in the symbols $\{x_i\}_{i=1}^{|S_u|}$ where $x_i$ belongs to the $i$th circle.
Here is the dictionary relating these three ways of describing the distinguished generators:
\begin{itemize}
    \item circle $i$ is labeled $-$ in the marked resolution
        = $x_-$ in position $i$ in the pure tensor
        = $x_i$ appears in the monomial
    \item circle $i$ is labeled $+$ in the marked resolution
        = $x_+$ in position $i$ in the pure tensor
        = $x_i$ does not appear in the monomial
\end{itemize}
Marked resolutions are useful in diagrammatic computations; 
Grigsby-Licata-Wehrli use the pure tensors in \cite{GLW-dt};
\SSS\ use monomials in \cite{SSS}.

Let $\vec{x}$ be a pure tensor in $\Kc(\D(L))$, located above vertex $u \in \boolcube^n$. Let $n_+$ (resp.\ $n_-$) be the number of positive (resp.\ negative) crossings in the oriented link diagram $\D$. 
\begin{itemize}
\item The \emph{homological grading} $\gr_h(\vec{x})$ is given by $|u| - n_-$.
\item The \emph{quantum grading} $\gr_q(\vec{x})$ is given by $|u| + \#(x_+) - \#(x_-) + n_+ - 2n_-$. 
\end{itemize}

% annular situation
In the presence of the basepoints $\mathbb{X}$ and $\mathbb{O}$, we tweak the definition above to reflect the position of the circles with respect to these basepoints. 
In a given complete resolution $\D_u$, the circles $S_u$ fall into two categories: those that separate the basepoints, and those that do not. We refer to the former as \emph{nontrivial} circles, and the latter as \emph{trivial} circles. To notationally distinguish these circles in the pure tensor setting, in place of $x_\pm$ in the previous paragraphs, we write $v_\pm$ for the $\pm$-labeling of a nontrivial circle and $w_\pm$ for the $\pm$-labeling of a trivial circle. (The quantum grading does not depend on the triviality of circles, e.g.\ $\#(x_-) = \#(v_-) + \#(w_-)$.)

We can now associate a third grading to the distinguished generators. 
\begin{itemize}
    \item The \emph{$k$-grading} $\gr_k(\vec{x})$ is given by
        $\#(v_+) - \#(v_-)$.
\end{itemize}

The Khovanov differential preserves the homological and quantum gradings, producing a bigraded homology theory. In the annular context, the differential is filtered (non-increasing) with respect to the $k$-grading.

\subsection{\Szabo's geometric spectral sequence}

In \cite{Szabo-gss}, \Szabo\ introduced his geometric spectral sequence in Khovanov homology. The underlying chains are the same as those described in Section \ref{subsec:Khovanov-homology}. The differentials are described in terms of \emph{resolution configurations} of index $\geq 1$, described below. Resolution configurations describe individual components of the differential by encoding two distinguished generators -- one ``before" picture and one ``after" picture -- and what arcs are used to surger to the ``before" picture to arrive at the ``after" picture. If a resolution configuration displays the ordered pair of distinguished generators (before, after) = $(\vec{x}, \vec{y})$, then $\vec{y}$ appears in the image of $\vec{x}$ under the differential. 
Figure \ref{fig:Szaboconfigurations} illustrates the five types of resolution configurations that contribute to the Szabo differential.

\begin{figure}
    \centering
    \begin{subfigure}{0.5\textwidth}
    \centering
    \includegraphics[width=.8\textwidth]{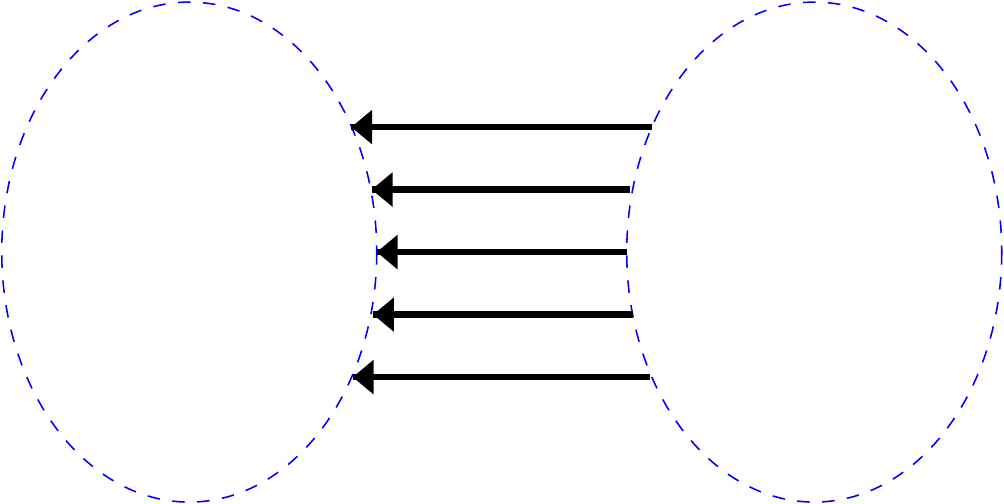}
    \caption{Type A}
    \label{fig:typeA}
    \end{subfigure}%
    \begin{subfigure}{0.5\textwidth}
    \centering
    \includegraphics[width=0.8\linewidth]{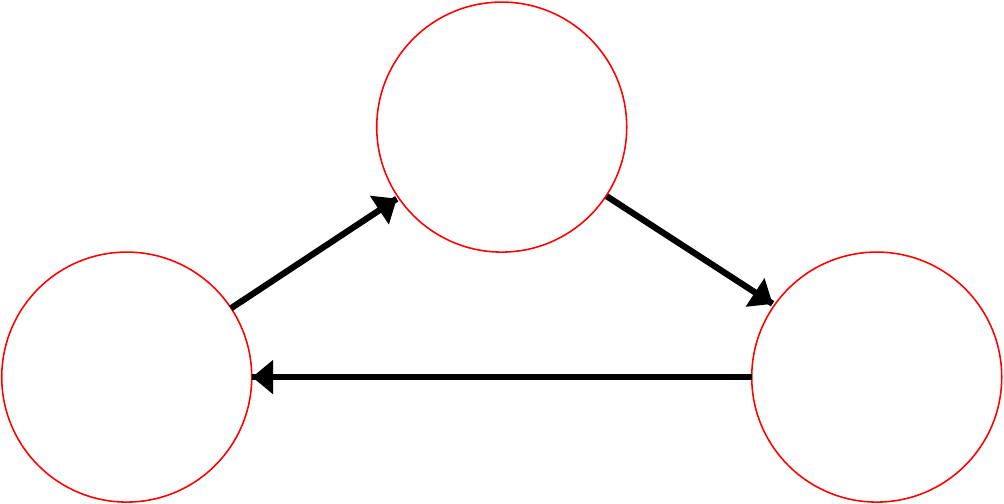}
    \caption{Type B}
    \label{fig:typeB}
    \end{subfigure}%
    \\
    \begin{subfigure}{0.5\textwidth}
    \centering
    \includegraphics[width=\linewidth]{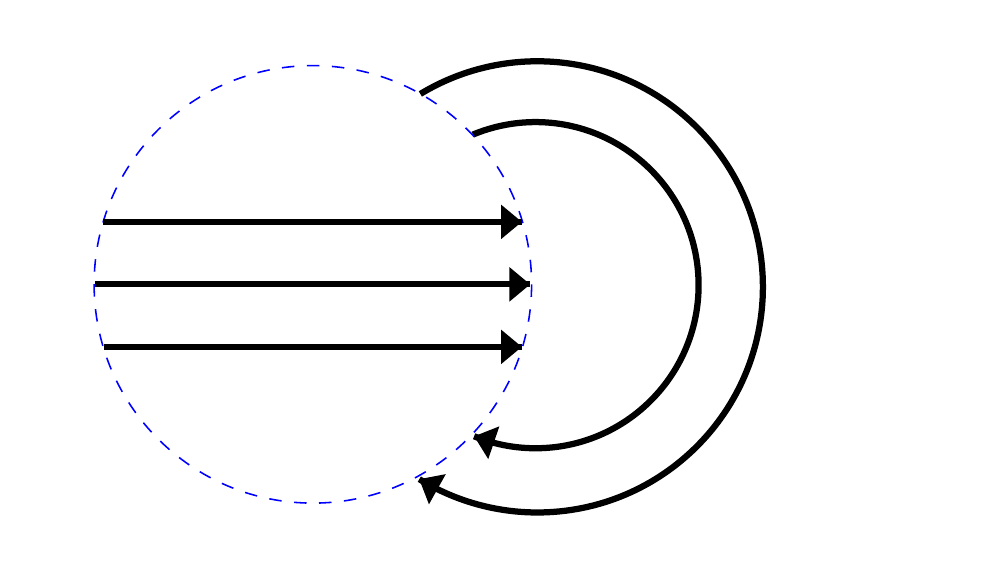}
    \caption{Type C}
    \label{fig:typeC}
    \end{subfigure}%
    \begin{subfigure}{0.5\textwidth}
    \centering
    \includegraphics[width=0.8\linewidth]{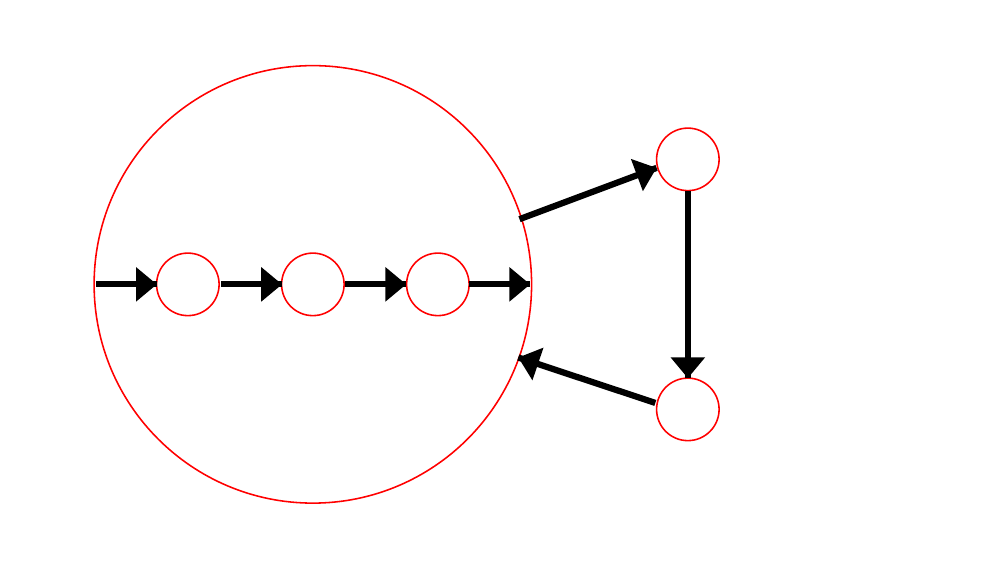}
    \caption{Type D}
    \label{fig:typeD}
    \end{subfigure}%
    \\
    \begin{subfigure}{0.5\textwidth}
    \centering
    \includegraphics[width=0.8\linewidth]{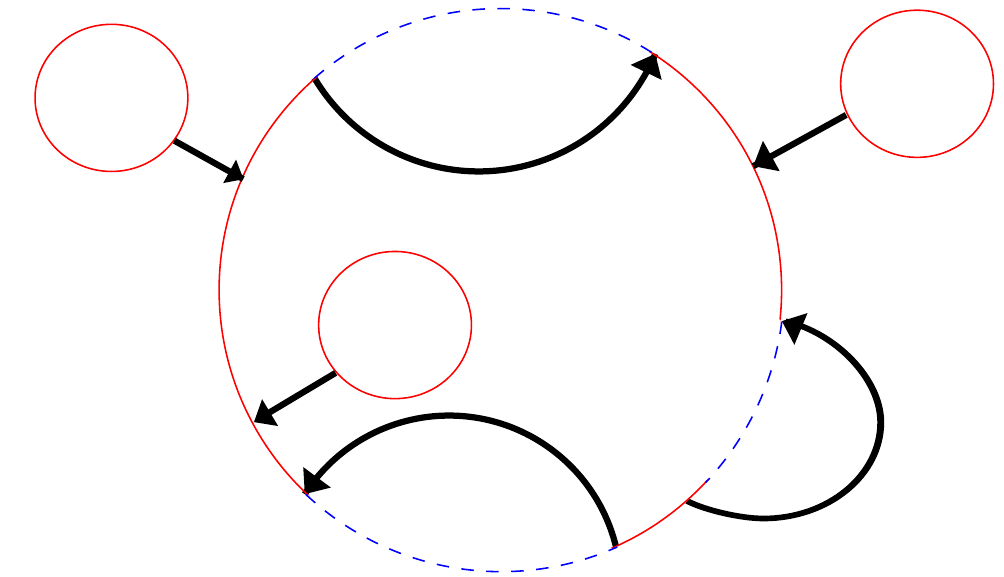}
    \caption{Type E}
    \label{fig:typeE1}
    \end{subfigure}%
    \begin{subfigure}{0.5\textwidth}
    \centering
    \includegraphics[width=0.8\linewidth]{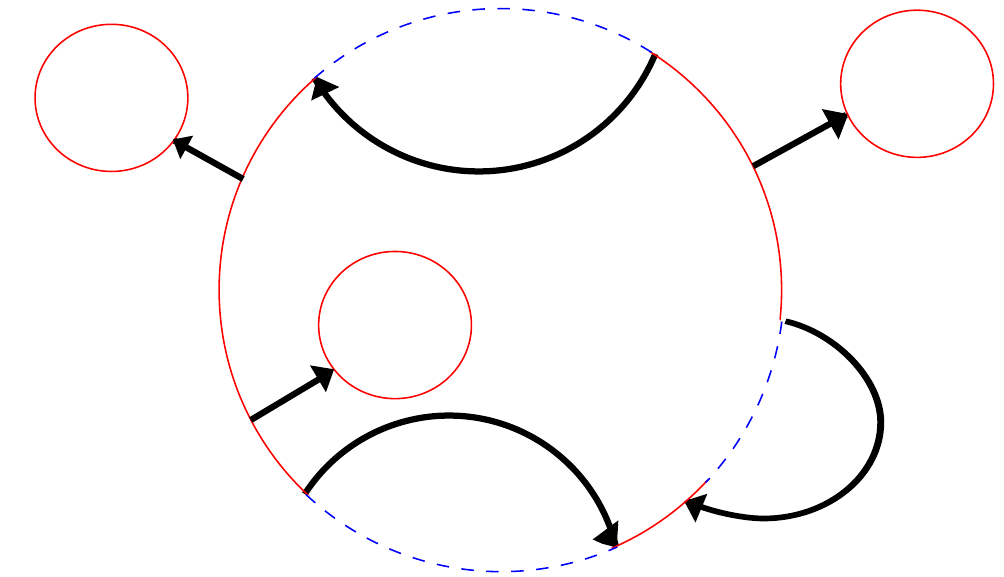}
    \caption{Type E}
    \label{fig:typeE2}
    \end{subfigure}%
    \caption{The oriented configurations that contribute to the Szab\'o differential}
    \label{fig:Szaboconfigurations}
\end{figure}

\Szabo's theory is $\gr_h$-filtered (cf.\ Proposition \ref{lem:tridegree}). The index 1 resolution configurations correspond to the original Khovanov differential. Thus $\gr_h$-filtration gives rise to the ``geometric spectral sequence" from Khovanov homology.

\subsection{The \SSS\ perturbation of the geometric spectral sequence}
\label{subsec:sss-defn}

The \SSS\ perturbation adds further differentials to the \Szabo\ complex, inspired by Bar-Natan's deformation of Khovanov homology \cite{BN-tangles-cobordisms}. The index 1 resolution configurations correspond to the original Bar-Natan deformation differentials. 

The total chain complex of Sarkar-Seed-Szab\'{o} \cite{SSS} is
\[\mathcal{C}_\tot = (\mathcal{C} \otimes \F[H,W], \delta_\tot = d_1 + W d_2 + W^2 d_3 + \cdots + H h_1 + HW h_2 + HW^2 h_3 + \cdots)\]
where the chain groups $\CC$ are the same as the Khovanov chain groups and the total differential lies in (homological, quantum)-grading (1,0). The maps $d_i$ correspond to the Szab\'o differential maps, and we refer the reader to \cite{SSS} for the definition of the maps $h_i$.  The Bar-Natan deformation of the Khovanov chain complex can be obtained from $\mathcal{C}_\tot$ by setting W = 0, and the Szab\'{o} chain complex can be obtained  from  $\mathcal{C}_\tot$ by setting $H = 0$. In this paper, we are mainly interested in the filtered version $(\CC_\ftot, \delta_\ftot)$:
\[\CC_\ftot = \CC_\tot/\{H=W=1\} = ( \mathcal{C}, d_1 + d_2 + d_3 + \cdots + h_1 + h_2 + h_3 + \cdots).\]

\section{Annular \SSS\ invariants}

\subsection{The annular filtration}
\label{subsec:annular-filtration}

From the data of the diagram $\D(L) \subset S^2 \setminus \{ \BBO, \BBX\}$ of an oriented annular link $L \subset A \times I$, we will use the construction of Sarkar-Seed-Szab\'{o} \cite{SSS}, building on  constructions of Khovanov \cite{Khovanov-00}, Bar-Natan \cite{BN-tangles-cobordisms}, and Szab\'{o} \cite{Szabo-gss}, to define a $\Z \oplus \Z \oplus \Z$-filtered chain complex as follows.

% Recall that we use $\mathbb{F} = \mathbb{F}_2$ coefficients throughout. 

\iffalse %%%% START COMMENTED OUT
Begin by temporarily forgetting the data of $\BBX$ and construct the standard Sarkar-Seed-Szab\'{o} complex $(\CC, \delta_\tot)$ associated to $\D(L)$ as described in \cite{SSS}. 
That is, choose an ordering of the crossings of $\D(L)$ and form the  \emph{cube of resolutions} of $\D(L)$ as described in [Kh].  This cube of resolutions determines a  chain complex over $\F[H,W]$
\[\CC = \bigoplus_{i,j \in \Z}\CC^{i,j}\]
along with an endomorphism called the total differential, $\delta_\tot : \mathcal{C} \to \mathcal{C}$,
\[\delta_\tot = d_1 + W d_2 + W^2 d_3 + \cdots + H h_1 + HW h_2 + HW^2 h_3 + \cdots\]
with the $H$ and $W$ being formal variables lying in $(\gr_h,\gr_q)$-grading $(0,-2)$ and $(-1,-2)$ and the differential  $\delta_\tot$ being of grading $(1,0)$, where $\gr_h$ is the homological grading and $\gr_q$ is the quantum grading. 

If we set $H=W=1$, we obtain a filtered version of this bicomplex, denoted $(\CC_\ftot, \delta_\ftot)$, where
\[\CC_\ftot = \CC_\tot/\{H=W=1\} = ( \mathcal{C}, d_1 + d_2 + d_3 + \cdots + h_1 + h_2 + h_3 + \cdots).\]
where the total differential $\delta_\ftot$ is nonnegative with respect to both $\gr_h$ and $\gr_q$. 
\fi %%%%%% END COMMENTED OUT

We obtain a third grading (the $k$ grading) on the Sarkar-Seed-Szab\'o chain complex $\CC_{tot}$, as follows. In the language of \cite{SSS}, a ``$0$'' (respectively, ``$1$'') marking on a starting circle denotes that the circle appears (respectively, does not appear) in the starting monomial. In a break with their conventions, we write $+$ and $-$ instead of the ``$0$" and ``$1$" that Sarkar-Seed-Szab\'o use on starting circles. (This change is made to reserve the use of numbers for the $k$ grading). We keep the \cite{SSS} notation for the ending monomial: A dashed blue (respectively, solid red) ending circle in a complete resolution of $\D(L)$ denotes that the ending circle does not appear (respectively, does appear) in the ending monomial. Recall that a circle are either \emph{trivial} (respectively, \emph{nontrivial}) if it intersects any fixed oriented arc $\gamma$ from $\BBX$ to $\BBO$ in an even (respectively, odd) number of points. We define the $k$ grading of a basis element of $\CC$ (or a square-free monomial in circles in a complete resolution) to be the number of nontrivial circles not appearing in the monomial minus the number of nontrivial circles appearing in the monomial. 

\begin{lemma}
\label{lem:tridegree}
The (filtered) total differential $\delta_\ftot$ decomposes into $(\gr_h,\gr_q,\gr_k)$-homogeneous pieces. In particular, 
\begin{itemize}
\item $d_i$ increases $\gr_h$ and $\gr_q$ by $i$ and $2i-2$, respectively, and decomposes into three $\gr_k$-homogeneous summands $d_{i,-2}, d_{i,0},$ and $d_{i,2}$, corresponding to a shift in $\gr_k$ by an element of $-2$, $0$, and $2$, respectively. Furthermore,  $d_{1,2} = 0$.  
\item $h_i$ increases $\gr_h$ and $\gr_q$ by $i$ and $2i$, respectively, and decomposes into $\gr_k$-homogeneous summands $h_{i, j}$ corresponding to a shift in $\gr_k$ by an even integer $j \in [0, i+1]$.  % [LT] although h can shift \gr_k by an arbitrarily high even integer, h_i can shift \gr_k by only up to i+1
\end{itemize}
Therefore we may write $\delta_\ftot = \sum_i (d_i + h_i)$, where  the  $d_i$ and $h_i$ decompose  into summands
\begin{itemize} 
\item $d_i = d_{i,-2} + d_{i,0} + d_{i, 2}$, with  $d_{1,2} = 0$
\item $h_i = h_{i, 0} + h_{i, 2} + h_{i, 4} + \cdots + h_{i, i+1}$ % although h can shift \gr_k by an arbitrarily high even integer, h_i can shift \gr_k by only up to i +1
\end{itemize}
and each summand is $(\gr_h,\gr_q,\gr_k)$-homogeneous.

\begin{proof}
The $\gr_h$- and $\gr_q$-homogeneity of $d_i$ and $h_i$ are given in \cite{SSS}. We will check the $k$-grading ($\gr_k$) shifts of $d_i$ and $h_i$. First, note that $d_1$ coincides precisely with the Khovanov differential. In \cite[Lemma 3]{GLW-dt} it is shown that $d_1$ splits into homogeneous summands that shift the $k$-grading by either $0$ or $-2$. Thus, we have $d_{1,2} = 0$ as claimed.

We next check the $k$-grading $\gr_k$ shift by $d_i$. We will consider all possible configurations of the Szab\'o differential (see Figure \ref{fig:Szaboconfigurations}). For each possible configuration and possible placement of the basepoints $\BBX$ and $\BBO$, we compute the $k$-grading difference of the ending and starting monomial associated to the configuration. If $\BBX$ is located in the same region of the plane as $\BBO$, then the $k$-grading shift is zero. In the following casework, we leave out this possibility. 

\begin{enumerate}
    \item In a Type-A configuration (Figure \ref{fig:typeA}), there are two starting circles, with some (at least one) parallel arcs between them. 
    Note that after removing the two starting circles from $S^2$, there are three connected components: two disks and an annulus. Up to symmetry, the basepoint $\BBO$ is either inside a disk region bounded by a starting circle (Figure \ref{fig:type-a2}) or inside the unique annulus bounded by the two starting circles (Figure \ref{fig:type-a1}). In every possible type A configuration, the $k$-grading shift is either $-2$, $0$, or $2$. See Figure \ref{fig:typeAx}.
    
    \begin{figure}
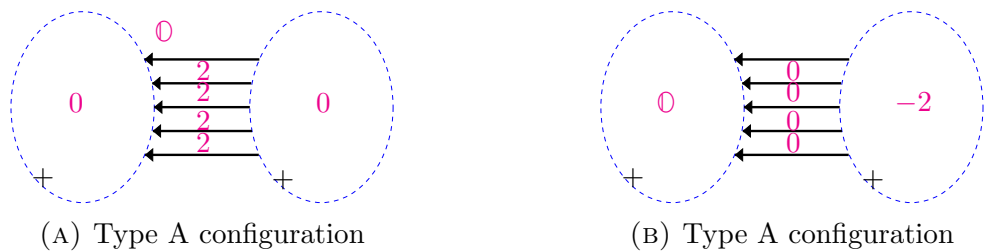

        \begin{subfigure}{0.45\textwidth}
        \centering
        \labellist
        \small
        \pinlabel	\textcolor{magenta}{$\mathbb{O}$}	at 	116	129.6
        \pinlabel	$\textcolor{magenta}{0}$	at 	49.6	76
        \pinlabel	$\textcolor{magenta}{2}$	at 	145.6	84
        \pinlabel	$\textcolor{magenta}{2}$	at 	145.6	100
        \pinlabel	$\textcolor{magenta}{2}$	at 	145.6	62.4
        \pinlabel	$\textcolor{magenta}{2}$	at 	145.6	45.6
        \pinlabel	$\textcolor{magenta}{0}$	at 	236	76
        %\pinlabel	$\textcolor{magenta}{0}$	at 	145.6	25.6 %X same as O
        \pinlabel	$+$	at 	24.8	19.2
        \pinlabel	$+$	at 	205.6	17.6
        \endlabellist
        \includegraphics[scale=.5]{type-a}
        \caption{Type A configuration}
        \label{fig:type-a1}
        \end{subfigure}%
        \quad 
        \begin{subfigure}{0.45\textwidth}
        \centering
        \labellist
        \small
        %\pinlabel	\textcolor{magenta}{$\mathbb{O}$}	at 	30.2	107.2 %X same as O
        \pinlabel	$\textcolor{magenta}{\mathbb{O}}$	at 	49.6	76
        \pinlabel	$\textcolor{magenta}{0}$	at 	145.6	84
        \pinlabel	$\textcolor{magenta}{0}$	at 	145.6	100
        \pinlabel	$\textcolor{magenta}{0}$	at 	145.6	62.4
        \pinlabel	$\textcolor{magenta}{0}$	at 	145.6	45.6
        \pinlabel	$\textcolor{magenta}{-2}$	at 	236	76
        \pinlabel	$+$	at 	24.8	19.2
        \pinlabel	$+$	at 	205.6	17.6
        \endlabellist
        \includegraphics[scale=.5]{type-a}
        \caption{Type A configuration}
        \label{fig:type-a2}
        \end{subfigure}
        \caption{Type A configurations with $k$-grading shifts corresponding to the placement of $\BBX$ basepoint}
        \label{fig:typeAx}
    \end{figure}
    \item A Type-B configuration (Figure \ref{fig:typeB}) is the dual of a Type-A configuration. Up to symmetry, there are two possibilities for the placement of $\BBO$: either outside the starting circles (Figure \ref{fig:type-b2}) or inside a starting circle (Figure \ref{fig:type-b1}). In every possible type B configuration, the $k$-grading shift is either $-2$, $0$, or $2$. See Figure \ref{fig:typeBx}.
    
    \begin{figure}
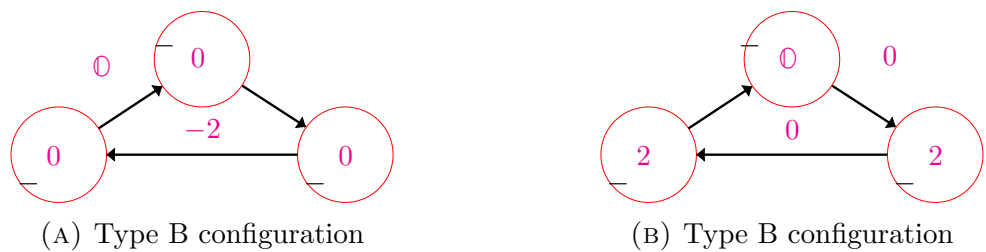

        \begin{subfigure}{0.45\textwidth}
        \centering
        \labellist
        \small
        \pinlabel	\textcolor{magenta}{$\mathbb{O}$}	at 	69	103
        \pinlabel	$\textcolor{magenta}{0}$	at 	141.6	109.6
        \pinlabel	$\textcolor{magenta}{0}$	at 	32.8	36
        \pinlabel	$\textcolor{magenta}{0}$	at 	252.8	36
        %\pinlabel	$\textcolor{magenta}{0}$	at 	218.4	111.2 %X same as O
        \pinlabel	$\textcolor{magenta}{-2}$	at 	144.8	55.2
        \pinlabel	$-$	at 	229.6	14.4
        \pinlabel	$-$	at 	13.6	14.4
        \pinlabel	$-$	at 	115.8	118.4
        \endlabellist
        \includegraphics[scale=.5]{type-b}
        \caption{Type B configuration}
        \label{fig:type-b2}
        \end{subfigure}%
        \quad
        \begin{subfigure}{0.45\textwidth}
        \centering
        \labellist
        \small
        \pinlabel	\textcolor{magenta}{$\mathbb{O}$}	at 	141.6 	109.6
        %\pinlabel	$\textcolor{magenta}{0}$	at 	147.2	128.8  %X same as O
        \pinlabel	$\textcolor{magenta}{2}$	at 	32.8	36
        \pinlabel	$\textcolor{magenta}{2}$	at 	252.8	36
        \pinlabel	$\textcolor{magenta}{0}$	at 	218.4	111.2
        \pinlabel	$\textcolor{magenta}{0}$	at 	144.8	55.2
        \pinlabel	$-$	at 	229.6	14.4
        \pinlabel	$-$	at 	13.6	14.4
        \pinlabel	$-$	at 	112.8	118.4
        \endlabellist
        \includegraphics[scale=.5]{type-b}
        \caption{Type B configuration}
        \label{fig:type-b1}
        \end{subfigure}
        \caption{Type B configurations with $k$-grading shifts}
        \label{fig:typeBx}
    \end{figure}
    \item In a Type-C configuration (Figure \ref{fig:typeC}), up to symmetry, there are two possibilities for the placement of $\BBO$ (Figure \ref{fig:type-c-symm}): the basepoint $\BBO$ is either between two parallel arcs or not. 
    For each possible type C configuration, the $k$-grading shift is either $0$ or $2$, depending on the location of $\BBX$. See Figure \ref{fig:typeCx}.
    \begin{figure}
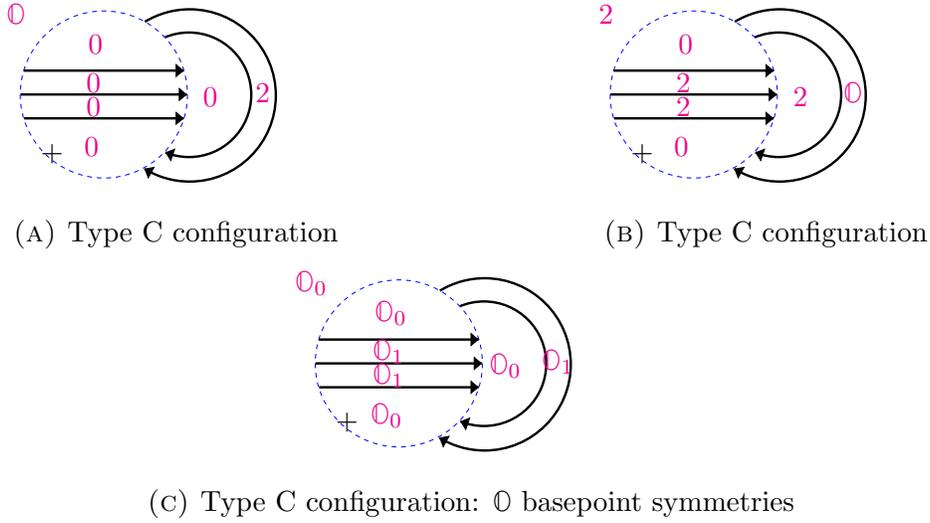

        \begin{subfigure}{0.45\textwidth}
        \centering
        \labellist
        \small
        \pinlabel	\textcolor{magenta}{$\mathbb{O}$}	at 	24	141.6
        \pinlabel	\textcolor{magenta}{${0}$}	at 	84	119.2
        \pinlabel	\textcolor{magenta}{${0}$}	at 	80.8	41.6
        \pinlabel	\textcolor{magenta}{${0}$}	at 	170.4	79.2
        \pinlabel	\textcolor{magenta}{${0}$}	at 	82.4	89.6
        \pinlabel	\textcolor{magenta}{${0}$}	at 	82.4	72
        \pinlabel	\textcolor{magenta}{${2}$} 	at 	210	82.4
        \pinlabel	$+$	at 	51.2	36.8
        \endlabellist
        \includegraphics[scale=.5]{type-c}
        \caption{Type C configuration }
        \label{fig:type-c0}
        \end{subfigure}%
        \quad
        \begin{subfigure}{0.45\textwidth}
        \centering
        \labellist
        \small
        \pinlabel	\textcolor{magenta}{$\textcolor{magenta}{2}$}	at 	24	141.6
        \pinlabel	\textcolor{magenta}{$\textcolor{magenta}{0}$}	at 	84	119.2
        \pinlabel	\textcolor{magenta}{$\textcolor{magenta}{0}$}	at 	80.8	41.6
        \pinlabel	\textcolor{magenta}{$\textcolor{magenta}{2}$}	at 	170.4	79.2
        \pinlabel	\textcolor{magenta}{$\textcolor{magenta}{2}$}	at 	82.4	89.6
        \pinlabel	\textcolor{magenta}{$\textcolor{magenta}{2}$}	at 	82.4	72
        \pinlabel	\textcolor{magenta}{$\mathbb{O}$} 	at 	210	82.4
        \pinlabel	$+$	at 	51.2	36.8
        \endlabellist
        \includegraphics[scale=.5]{type-c}
        \caption{Type C configuration }
        \label{fig:type-c1}
        \end{subfigure}%
        \\
        \begin{subfigure}{\textwidth}
        \centering
        \labellist
        \small
        \pinlabel	\textcolor{magenta}{$\mathbb{O}_0$}	at 	24	141.6
        \pinlabel	\textcolor{magenta}{$\mathbb{O}_0$}	at 	84	119.2
        \pinlabel	\textcolor{magenta}{$\mathbb{O}_0$}	at 	80.8	41.6
        \pinlabel	\textcolor{magenta}{$\mathbb{O}_0$}	at 	170.4	79.2
        \pinlabel	\textcolor{magenta}{$\mathbb{O}_1$}	at 	82.4	89.6
        \pinlabel	\textcolor{magenta}{$\mathbb{O}_1$}	at 	82.4	72
        \pinlabel	\textcolor{magenta}{$\mathbb{O}_1$} 	at 	210	82.4
        \pinlabel	$+$	at 	51.2	36.8
        \endlabellist
        \includegraphics[scale=.5]{type-c}
        \caption{Type C configuration: $\mathbb{O}$ basepoint symmetries}
        \label{fig:type-c-symm}
        \end{subfigure}%
        \caption{Type C configurations with $k$-grading shifts}
        \label{fig:typeCx}
    \end{figure}
    \item A Type-D configuration (Figure \ref{fig:typeD}) is the mirror of the dual of a Type-C configuration. Thus, the $k$-grading shift candidates are the same as those of a type C configuration: $0$ or $2$.
    
    \iffalse
    \begin{figure}
        \begin{subfigure}{0.5\textwidth}
        \centering
        \includegraphics[width=.5\linewidth]{typeD1.png}
        \caption{Type-D with $\BBX$ inside a starting circle}
        \label{fig:D1}
        \end{subfigure}%
        \begin{subfigure}{0.5\textwidth}
        \centering
        \includegraphics[width=.5\linewidth]{typeD2.png}
        \caption{Type-D with $\BBX$ inside a starting circle}
        \label{fig:D2}
        \end{subfigure}
        \\
        \begin{subfigure}{0.5\textwidth}
        \centering
        \includegraphics[width=.5\linewidth]{typeD3.png}
        \caption{Type-D with $\BBX$ outside a starting circle}
        \label{fig:D3}
        \end{subfigure}%
        \begin{subfigure}{0.5\textwidth}
        \centering
        \includegraphics[width=.5\linewidth]{typeD4.png}
        \caption{Type-D with $\BBX$ inside a starting circle}
        \label{fig:D4}
        \end{subfigure}
        \caption{Type D configurations with a basepoint $\BBX$}
        \label{fig:typeDx}
    \end{figure}
    \fi
    \item In a Type-E configuration (Figure \ref{fig:typeE1} or Figure \ref{fig:typeE2}), up to symmetry there are three possibilities for the location of the $\BBO$ basepoint: see Figure \ref{fig:type-e-symm}. For each possible $\BBX$ and $\BBO$ pairing, the $k$-grading shift is either $-2$, $0$, or $2$. See Figure \ref{fig:typeEx}.
        
    \begin{figure}[ht]
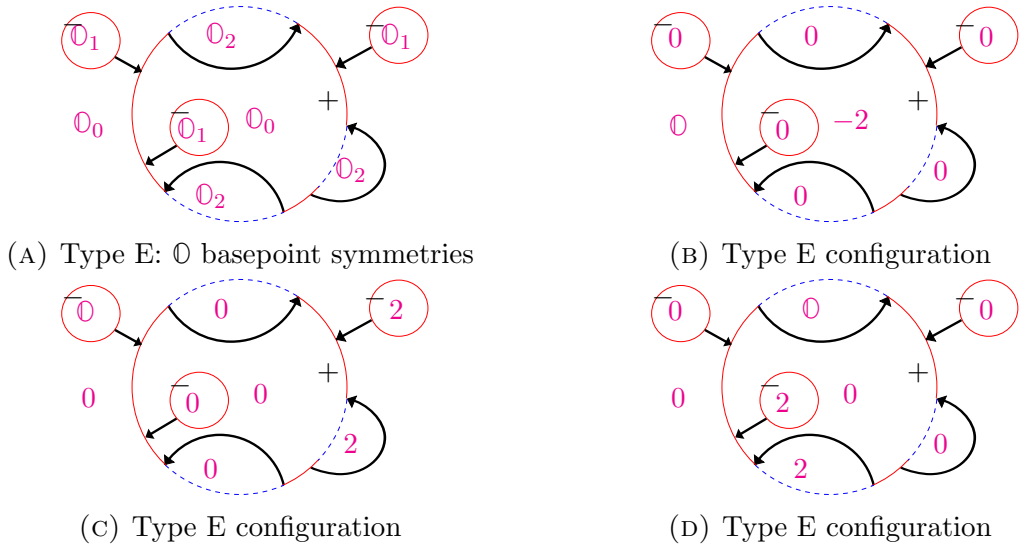

    \begin{subfigure}{0.45\textwidth}
        \centering
        \labellist
        \small
        \pinlabel	\textcolor{magenta}{$\mathbb{O}_0$}	at 	30.4	74.4
        \pinlabel	\textcolor{magenta}{$\mathbb{O}_0$}	at 	160	    77.6
        \pinlabel	\textcolor{magenta}{$\mathbb{O}_2$}	at 	130.4	141.6
        \pinlabel	\textcolor{magenta}{$\mathbb{O}_2$}	at 	228	    40
        \pinlabel	\textcolor{magenta}{$\mathbb{O}_2$}	at 	122.4	20.8
        \pinlabel	\textcolor{magenta}{$\mathbb{O}_1$}	at 	108.8	70.4
        \pinlabel	\textcolor{magenta}{$\mathbb{O}_1$}	at 	28	    140
        \pinlabel	\textcolor{magenta}{$\mathbb{O}_1$}	at 	262.4	140.8
        \pinlabel	$+$	at 	211.2	92.8
        \pinlabel	$-$	at 	246.4	148.8
        \pinlabel	$-$	at 	99.4	83.2
        \pinlabel	$-$	at 	18.64	149.6
        \endlabellist
        \includegraphics[scale=.5]{type-e1}
        \caption{Type E: $\mathbb{O}$ basepoint symmetries}
        \label{fig:type-e-symm}
    \end{subfigure}%
    \quad
    \begin{subfigure}{0.45\textwidth}
        \centering
        \labellist
        \small
        \pinlabel	\textcolor{magenta}{$\mathbb{O}$}	at 	30.4	74.4
        \pinlabel	\textcolor{magenta}{$-2$}	at 	160	77.6
        \pinlabel	\textcolor{magenta}{$0$}	at 	130.4	141.6
        \pinlabel	\textcolor{magenta}{$0$}	at 	228	40
        \pinlabel	\textcolor{magenta}{$0$}	at 	122.4	20.8
        \pinlabel	\textcolor{magenta}{$0$}	at 	108.8	70.4
        \pinlabel	\textcolor{magenta}{$0$}	at 	28	140
        \pinlabel	\textcolor{magenta}{$0$}	at 	262.4	140.8
        \pinlabel	$+$	at 	211.2	92.8
        \pinlabel	$-$	at 	246.4	148.8
        \pinlabel	$-$	at 	99.4	83.2
        \pinlabel	$-$	at 	18.64	149.6
        \endlabellist
        \includegraphics[scale=.5]{type-e1}
        \caption{Type E configuration}
        \label{fig:type-e1}
    \end{subfigure} 
    \\
    \begin{subfigure}{0.45\textwidth}
        \centering
        \labellist
        \small
        \pinlabel	\textcolor{magenta}{$0$}	at 	30.4	74.4
        \pinlabel	\textcolor{magenta}{$0$}	at 	160	77.6
        \pinlabel	\textcolor{magenta}{$0$}	at 	130.4	141.6
        \pinlabel	\textcolor{magenta}{$2$}	at 	228	40
        \pinlabel	\textcolor{magenta}{$0$}	at 	122.4	20.8
        \pinlabel	\textcolor{magenta}{$0$}	at 	108.8	70.4
        \pinlabel	\textcolor{magenta}{$\mathbb{O}$}	at 	28	140
        \pinlabel	\textcolor{magenta}{$2$}	at 	262.4	140.8
        \pinlabel	$+$	at 	211.2	92.8
        \pinlabel	$-$	at 	246.4	148.8
        \pinlabel	$-$	at 	99.4	83.2
        \pinlabel	$-$	at 	18.64	149.6
        \endlabellist
        \includegraphics[scale=.5]{type-e1}
        \caption{Type E configuration}
        \label{fig:type-e2}
    \end{subfigure}%
    \quad
    \begin{subfigure}{.45\textwidth}
        \centering
        \labellist
        \small
        \pinlabel	\textcolor{magenta}{$0$}	at 	30.4	74.4
        \pinlabel	\textcolor{magenta}{$0$}	at 	160	77.6
        \pinlabel	\textcolor{magenta}{$\mathbb{O}$}	at 	130.4	141.6
        \pinlabel	\textcolor{magenta}{$0$}	at 	228	40
        \pinlabel	\textcolor{magenta}{$2$}	at 	122.4	20.8
        \pinlabel	\textcolor{magenta}{$2$}	at 	108.8	70.4
        \pinlabel	\textcolor{magenta}{$0$}	at 	28	140
        \pinlabel	\textcolor{magenta}{$0$}	at 	262.4	140.8
        \pinlabel	$+$	at 	211.2	92.8
        \pinlabel	$-$	at 	246.4	148.8
        \pinlabel	$-$	at 	99.4	83.2
        \pinlabel	$-$	at 	18.64	149.6
        \endlabellist
        \includegraphics[scale=.5]{type-e1}
        \caption{Type E configuration}
        \label{fig:type-e3}
    \end{subfigure}  
    \caption{Type E configurations with $k$-grading shifts}
    \label{fig:typeEx}
    \end{figure} 

\end{enumerate}

%To see that $\gr_k$ only ever changes in even number increments, consider the mod 2 nesting of circles. If one draws a line from $\BBO$ to $\BBX$, a surgery either adds two intersections or doesn't add any. In any case, the mod 2 intersection number between the line from $\BBO$ to $\BBX$ and the circles does not change.

To see that $\gr_k$ only ever changes in even number increments, consider the mod 2 nesting of circles. Observe that, with respect to an arc $\gamma$ drawn from $\BBO$ to $\BBX$, a surgery does not change the mod 2 intersection number between $\gamma$ and the circles.

\iffalse
\begin{figure}
\includegraphics[width=2in]{kdeg_h.png}
\caption{A dual tree configuration which contributes to $h$. The red numbers indicate the $k$-degree shift (of $h_i$ for some $i$) corresponding to a diagram where $\BBX$ is located in the region of the red number.}
\end{figure}
\fi

To see the $k$-grading shifts of $h_i$, recall that the  resolution configurations $(R, x, y)$ contributing to $h_i$ are disjoint unions of trees and dual trees (see \cite[Definition 3.2]{SSS}). We first focus on computing the $k$-grading shift of a single tree. Recall that a tree is a labeled resolution configuration of some index $i \geq 0$, with exactly
$(i+1)$ starting circles, exactly one ending circle, and no passive circles (unless $i=0$), and with all the starting
circles appearing in the starting monomial, and the ending circle appearing in the ending monomial. Therefore, the initial $k$-grading must be a non-positive integer in the interval $[-(i+1), 0]$, while the ending $k$-grading is either $0$ or $-1$. Moreover, it is easy to check that $k$-grading increases by twice the number of nontrivial node circles in the tree. Therefore, the $\gr_k$ shift must be an \emph{even} integer. Thus, an index $i$ tree configuration contributes to $h_i$ (increasing the homological grading $\gr_h$ by $i$) and must increase the $k$-grading by an even integer in $[0, i+1]$.

Next, observe that a configuration and its dual admit the same $k$-grading shifts. Therefore, the $k$-grading shift of a single dual tree of index $i$ is also an even integer in $[0, i+1]$. 

Finally we analyze the behavior of the $k$-grading shift under disjoint unions of trees and dual trees.
Given any resolution configuration $(R, x, y)$, we may find an ``equator" circle $C$ on the sphere which does not intersect any of the arcs or circles in the resolution configuration. This ``equator" circle $C$ does not belong to the circles in $(R, x, y)$. Moreover, $C$ divides the sphere into two disks $D_1$ and $D_2$ containing the two resolution configurations $(R_1, x_1, y_1)$ and $(R_2, x_2, y_2)$ respectively. We have $(R, x, y) = ( R_1 \sqcup R_2, x_1x_2, y_1y_2)$.
\begin{enumerate}
    \item If the basepoints $\BBO$ and $\BBX$ lie in the same disk, say $D_1$, then the $k$-grading shift of the resolution configuration $(R, x, y)$ is the same as the $k$-grading shift of a resolution configuration $(R_1, x_1, y_1)$. 
    \item If the basepoints $\BBO$ and $\BBX$ lie in the disks $D_1$ and $D_2$ respectively, then the overall $k$-grading shift of the resolution configuration $(R, x, y)$ equals the sum of the following two quantities:
    \begin{itemize}
        \item the $k$-grading shift $(R_1, x_1, y_1)$ where $\BBX$ is placed at ``the point at infinity" of $D_1$ 
        \item the $k$-grading shift $(R_2, x_2, y_2)$ where $\BBO$ is placed at ``the point at infinity" of $D_2$ 
    \end{itemize}
    This can be seen by dividing the arc from $\BBX$ to $\BBO$ into two pieces. 
\end{enumerate}
It follows that $h_i$ increases $\gr_k$ by an even integer in $[0, i+1]$.
\end{proof}
\end{lemma}

\begin{corollary}
The $\gr_h$, $\gr_q$, and $\gr_q-\gr_k$ gradings on $\CC$ endows $(\CC_{ftot}, \delta_{ftot})$ with the structure of a $(\Z \oplus \Z \oplus \Z)$-filtered complex. 
\begin{proof}
For each $(a,b,c) \in \Z \oplus \Z \oplus \Z$, define  
\[\CF_{a,b,c} = \text{Span} \{ x \in \CC \ | \ (\gr_h, \gr_q, \gr_{q-k})(x)  \geq (a,b, c)\}\]
Lemma \ref{lem:tridegree} tells us that $\delta_{ftot}$ is non-decreasing with respect to the $\gr_h$, $\gr_q$, and $\gr_q-\gr_k$ gradings.  So $\CF_{a,b,c}$ is a subcomplex for each $(a,b,c) \in (\Z \oplus \Z \oplus \Z)$. Moreover, $(a',b',c') \geq (a,b,c) \in (\Z \oplus \Z \oplus \Z)$ implies $\CF_{a',b',c'} \subset \CF_{a,b,c}$, as desired. 
\end{proof}
\end{corollary}

Having considered the $(\gr_h, \gr_q, \gr_k)$ triple degree of all the elementary cobordisms (see Lemma \ref{lem:tridegree}), we can now define a two-parameter family of gradings $\gr_{r,t}$.

\begin{definition}
Let $x$ be a $(\gr_h, \gr_k, \gr_j)$-homogeneous basis element of $\CC$. Let $t \in [0,1]$ and $r \in [0,1].$ 
Define \[\gr_{r,t}(x) = r \cdot \gr_h (x) + (1-r)(\gr_q (x)- t\cdot \gr_k(x)).\]
\end{definition}

We note that $r = 0$ corresponds to $\gr_{0,t} = \gr_q - t \cdot \gr_k$, which agrees with the $j_t$ grading in \cite{GLW-dt} (although their construction uses the Khovanov-Lee complex instead of the Sarkar-Seed-Szab\'o complex). 
The case $t = 1$ corresponds to $\gr_{r,1} = r \cdot \gr_h +(1-r) \cdot \gr_q$, which is precisely the filtration grading used in \SSS\ to define a family of concordance invariants that generalize the Rasmussen $s$ invariant.

\begin{corollary}
For every $r,t \in [0,1]$, the $\gr_{r,t}$ grading endows $(\CC, \delta_{ftot})$ with the structure of a (discrete, bounded) $\R$-filtered complex equipped with a finite filtered graded basis. 
\begin{proof}
Lemma \ref{lem:tridegree} implies that $\delta_{ftot}$ is non-decreasing with respect to the $\gr_{r,t}$ grading for each $r,t \in [0,1]$. It follows that for each $ a \in \R$, the subcomplexes 
\[\CF_a = \text{Span}\{x \in \CC \ | \ \gr_{r,t}(x) \geq a\}\]
endow $\CC$ with the structure of an $\R$-filtered complex. The finiteness of the distinguished filtered graded basis implies that the $\R$-filtration is discrete and bounded.  
\end{proof}
\end{corollary}

Therefore, $\CC$ admits a \emph{filtration grading} given by 
\[\grrt(x) = \max \{i \in I \mid x \in \mathcal{F}_i\}.\]
The homology $H_*(\CC)$ inherits a grading from $\CC$ as follows.
\begin{definition} Given $0 \neq [x] \in H_*(\CC)$, we define the $\grrt$ grading of $[x]$ to be
\[ \grrt([x]) = \max_{y \in \CC} \{ \grrt(y) \mid [y] = [x] \}. \]
\end{definition}

Let $L \subset (A \times I)$ be an $l$-component annular link. By Proposition 5.4 of \cite{SSS}, $H_*(\CC_\ftot(L))$ has rank $2^l$ and there are canonical generators $g(o)$ in canonical correspondence with orientations $o$ on $L$, just as in the the Bar-Natan \cite{turner-calculating} and Lee \cite{Lee-perturbation} theories. 
Given a link diagram $\D$ with an orientation, the canonical generator $g(o) \in \CC_\ftot$ is given by the tensor product of labels $x_i$ or $1 + x_i$ on the circles in the oriented resolution, assigned by the following rule. Fix the white-and-black checkerboard colroing of the oriented resolution where the unbounded region is white, and label circle $i$ with $x_i$ if it bounds a black region, or $1+x_i$ if it bounds a white region.

\begin{definition}
Let $r \in [0,1]$ and $t \in [0,1]$. 
Define $s_{r,t}(L,o) = \gr_{r,t}([g(o)])$.  
\end{definition}

We will show in part \ref{thm:annlinkinvt} of Theorem \ref{thm:main-thm} that the $\srt(L,o)$ invariants are indeed oriented annular link invariants. 
\begin{example}
We compute $\srt(\widehat{\mathbb{1}}_n,o)$ for the closure of the identity $n$-braid $\widehat{\mathbb{1}}_n$ with a choice of orientation $o$. Using the crossingless diagram for $\idclosure_n$, we observe that $s_{r,t}(\idclosure_n, o)  
= \gr_{r,t}(v_- \otimes v_- \otimes \cdots \otimes v_-)
= (1-r)(-n + tn) = -n(1-r)(1-t).$
In particular, $\srt(\idclosure_n, o)$ is independent of the choice of orientation $o$ of $\idclosure_n$.
\iffalse
\begin{figure}[htb!]
\begin{tikzpicture}
\draw (0,0) circle (10pt);
\draw (0,0) circle (14pt);
\draw (0,0) circle (18pt);
\node[color=magenta] at (0,0) (x) {$\BBX$};
\node[color=magenta] at (1,0) {$\BBO$};
\end{tikzpicture}
\caption{The identity $3$-braid closure $\widehat{\mathbb{1}}_3$}
\label{fig:idclosure}
\end{figure}
\fi
\end{example}

\begin{example}
Let $\idclosure^+_1 \in \mathfrak{B}_{2}$ denote the closure of the positive Markov stabilization of the identity $1$-braid. Given any choice of orientation $o$ on $\idclosure^+_1$,
\[\srt(\idclosure^+_1, o) = (1-r)(-1+2t)\]
Let $\idclosure^-_1 \in \mathfrak{B}_{2}$ denote the closure of the negative Markov stabilization of the identity $1$-braid. Given any choice of orientation $o$ on $\idclosure^-_1$,
\[\srt(\idclosure^-_1, o) = (1-r)(-1)\]
\end{example}

% regarding difference between upright sets and our filtration angles
\begin{remark}[Comparison with the \SSS\ concordance invariants]
\label{rmk:relate-sss-invts}
We regard the real-valued $s_{r,0}(K,o)$ invariant (where $t =0$) as a continuous sibling of the \SSS\ $2\Z$-valued $s_o^{\mathcal{U}_{(r)}}(K)$ concordance invariant as follows. For $r \in [0,1]$, \SSS\ define the \emph{upright set} 
\[
\mathcal{U}_{(r)} = \{ (h,q) \in \Z \times (2\Z \times 1) \st r \cdot h + (1-r) \cdot q > 0 \text{, or } [r \cdot h + (1-r) \cdot q = 0 \text{ and } q \geq 0 ]\}.
\]
\SSS\ introduce notation for shifts $\mathcal{U}[n]$ of an upright set $\mathcal{U}$ in the $(\gr_h, \gr_q)$-plane:
\[
(h,q) \in \mathcal{U}[n] \iff (h, q-n) \in \mathcal{U} 
\]
for any upright set $\mathcal{U}$. \SSS\ only allow shifts by even integers $n \in \Z$, but for purposes of comparison with our invariants, we extend the definition to allow for shifts by any real number $n \in \R$. \SSS\ define
\[
s_o^{\mathcal{U}} (K, o) = \max \{ n \in 2\Z \st \CF_{\mathcal{U}[n]} \CC_{\ftot}(K) \text{ contains a rep.\ of } [g(o)] \} + 2
\]
where $[g(o)]$ denotes the homology class of the generator $g(o)$.

In our case, the filtration level $\CF_n^{r,0}$ of $\CC_\ftot$ defined by the grading $\gr_{r,0}$ is the half-plane
\[
\CF_n^{r,0} = \{ (h,q) \in \R \times \R \st r \cdot h + (1-r) \cdot q \geq n\},
\]
where we abuse notation by conflating a filtration level of $\CC$ with the half plane in $(\gr_h, \gr_q)$-space it is defined to be supported on. We see that
\[
\CF_n^{r,0} \cap (\Z \times (2\Z+1)) 
= \mathcal{U}_{(r)} \left [\frac{n}{1-r} \right] \cup \{ (h, q) \mid  r\cdot h + (1-r)\cdot q = n \text{ and } (1-r)q < n\}.
\]
That is, our discretized filtration levels $\CF_n^{r,0} \cap (\Z \times (2\Z+1))$ agree with the \SSS\ filtration levels $\mathcal{U}_{(r)}[\frac{n}{1-r}]$, with a different boundary condition.

\begin{remark}
\label{rmk:non-annular-comparison}
If the annular link $L$ is unlinked from the unknotted axis $U$, then $\srt(L) = s_{r,0}(L)$ for all $t \in [0,1]$.
Furthermore, if $[L]$ denotes the isotopy class of $L$ in $S^3$, then $s_{r,0}(L, o)$ ``agrees" with $s^{\mathcal{U}_(r)}([L], o)$, by the relationship described in Remark \ref{rmk:relate-sss-invts}.
\begin{proof}
Observe that all distinguished generators have $\gr_k = 0$, in which case $\grrt = r \cdot \gr_h + (1-r) \gr_q$.
\end{proof}
\end{remark}

\iffalse
By definition, the line $\partial \mathcal{U}_{(r)} = \partial \CF_0^{r,0}$ contains the points $(0,0)$ and $(h',q')$ where $r \cdot h' + (1-r) \cdot q' = 0$.
The line shifted up by $n$ units (along the $\gr_q$-axis) contains the points $(0,n)$ and $(h', q'+n)$.
On the other hand, $\partial \CF_n^{r,0}$ is the line $r \cdot h + (1-r) \cdot q = n$, which contains the points $(0, \frac{n}{1-r})$ and $(\frac{1}{r}, 0)$.
\fi
\end{remark}

\subsection{Behavior under cobordisms}

Let us now review the definitions of the maps on $\CC_\ftot$ associated to the elementary movie moves, and describe their $(\gr_h, \gr_q, \gr_k)$ tridegree in an annular context. For the non-annular case, just consider the $(\gr_h, \gr_q)$ bidegree. Actual computations are performed in Section \ref{subsec:annular-filtration}, specifically in Lemma \ref{lem:tridegree}. 

In all of these cases, we consider a diagrammatic cobordism $F: \D \to \D'$. The cobordism map associated to the elementary move is a map $f: \CC_\ftot(\D) \to \CC_\ftot(\D')$.

First of all, for annular isotopies and annular Reidemeister moves, the associated map is the canonical isomorphism of the target and source complexes. Thus the interesting moves are the elementary cobordisms. 

To a birth, \SSS\ associate the map $a \mapsto a \otimes x_+$. Dually, to a death, \SSS\ associate the map $a \otimes x_- \mapsto a, a \otimes x_+ \mapsto 0$. In the annular case, let $x_\pm = v_\pm$ or $w_\pm$ depending on whether the circle is nontrivial or trivial, respectively. 

To an annular saddle, we get the associated map from looking at the differential in the complex corresponding the link diagram $\D''$ with a crossing at the site of the saddle cobordism, so that $\CC_\ftot(\D'') = \mathrm{cone}(\CC_\ftot(\D) \to \CC_\ftot(\D')[1]\{1\})$ (see Remark \ref{rmk:confusing-shifts}). So the cobordism map $f$ looks like the sum of the components of the differential (the $d_i$ and $h_i$), except that the homological and quantum degrees are one less than that of the true components of the differential. 
% If we abuse notation and write $f = \sum d_i + \sum h_i$, we see that the tridegree of the components is as follows.

\begin{remark}
\label{rmk:confusing-shifts}
The chain map associated to a saddle cobordism is defined as a map $\CC_\ftot(\D) \to \CC_\ftot(\D')$. Following \cite{BN-Kh-intro}, we use the convention that $\CC[a]\{b\}^{h,q} = \CC^{h-a, q-b}.$ We can visualize degree shift as follows: $\CC[a]\{b\}$ is the result of grabbing $\CC$ and moving it $a$ units along the $\gr_h$-axis and $b$ units along the $\gr_q$-axis. 
\end{remark}

We may now describe the filtration degrees of the elementary cobordism maps.

\begin{proposition}[cf.\ \cite{GLW-dt} Prop.\ 2]
\label{prop:cob-filt-degrees}
For $r,t \in [0,1]$, the $\gr_{r,t}$ filtration degrees of \SSS\ chain map associated to 
\begin{enumerate}
    \item an annular elementary saddle cobordism is $(1-r)(-1)$,
    \item an annular birth / death is $1-r$,
    \item a non-annular birth / death is $(1-r)(1-t)$, and 
    \item an annular Reidemeister move is $0$.
\end{enumerate}
\begin{proof}
First note that the homological filtration degree shift for all of these elementary cobordism maps is 0.  % explain why
In the following, ``bifiltration'' means $(\gr_q, \gr_k)$ bifiltration. 
Below, we compute the lower bound on the degree shift under elementary cobordism maps on $\gr_q - t \gr_k$, so that the lower bound on the degree shift for $\gr_{r,t}$ is $(1-r)$ times that of $\gr_q - t \gr_k$.

\begin{enumerate}

% saddle cobord
\item Recall that $d_\ftot = \sum_{i \geq 1} d_i + h_i$, where 
\begin{enumerate}
    \item $d_i = d_{i,-2} + d_{i,0} + d_{i, 2}$, with  $d_{1,2} = 0$
    \item $h_i = h_{i, 0} + h_{i, 2} + h_{i, 4} + \cdots + h_{i, i+1}$ 
\end{enumerate}
\begin{table}[h!]
\centering
 \begin{tabular}{||c |  c | c c  c|| c||} 
 \hline
 &  \ & $\gr_h$ &  $\gr_q$ &  $\gr_k$ &  $ \gr_q - t \gr_k$  \\ [0.5ex] 
 \hline\hline
$i \geq 1$  &   $d_{i, -2}$ & $i$ & $2i-2$ &  $-2$ & $2i-2 + 2t$ \\ 
$i \geq 1$  &  $d_{i, 0}$ & $i$ &  $2i-2$ &  $0$ & $2i-2 $ \\
$i >1$   & $d_{i, 2}$ & $i$ &  $2i-2$ &  $2$ &$2i-2 - 2t$  \\
$i \geq 1, \ j \in [0, i+1]$  & $h_{i, j}$ &       $i$  &  $2i$    & $j$ & $2i-jt$ \\
 \hline
 \end{tabular}
\end{table}
Therefore, the lower bound for the $\gr_q - t \gr_k$ shift is $0$. 

Note that the filtrations degrees in the chart above take into account the quantum grading shift of +1 when we move along an edge, within a chain complex.
The map associated to a saddle cobordism corresponds to such an edge map, but without this shift in quantum grading, so the lower bound for the $\gr_q - t \gr_k$ shift is $-1$ for a saddle cobordism.

% annular birth / death
\item Let $\mathbb{W}$ denote the 2-dimensional vector space over $\F_2$ generated by two distinguished generators, at $(\gr_q, \gr_k)$ degrees $(1,0)$ and $(-1,0)$. This is the vector space underlying the chain complex for a single annular circle.

The birth of an annular circle corresponds to the inclusion map $ \CC \to \CC \otimes \mathbb{W}$ which takes $x \mapsto x \otimes w_+$, and hence has bifiltration degree $(1,0)$. 

The death of an annular circle corresponds to the projection map $\CC \otimes \mathbb{W}$ which takes $x \otimes w_+ \mapsto 0$ and $x \otimes w_- \mapsto x$, and hence also have bifiltration degree $(1,0)$ as well. (Note that one can think of  $\gr_q(0 \in \F_2[x]/(x^2)) = \infty$.) % explain further

Therefore, the lower bound for the $\gr_q - t \gr_k$ shift is $1$. 

% nonannular birth/death
\item  Similarly, let $\mathbb{V}$ be the 2-dimensional vector space underlying the chain complex for a single non-annular circle. This has two distinguished generators, at bifiltration degrees  $(1,1)$ and $(-1,-1)$. 

The nonannular birth map  $ \CC \to \CC \otimes \mathbb{V}$ sends $x \mapsto x \otimes v_+$ and hence has bifiltration degree $(1,1)$.

The nonannular death map  $ \CC \otimes \mathbb{V} \to \CC$ sends $x \otimes v_+\mapsto 0$ and $x \otimes v_- \mapsto x$  and hence has bifiltration degree $(1,1)$ as well. 

Therefore, the lower bound for the $\gr_q - t \gr_k$ shift is $1-t$. 

% annular Reidemeister moves
\item
As in \cite{GLW-dt} Proposition 2, the chain homotopies that give Reidemeister equivalences do not interact with the annular axis, so the $\gr_k$ filtration degree of the annular cobordism (``annular'' in the sense of a concordance) is 0.

The $\gr_q$ filtration degree is 0 by Section 4 of \cite{SSS}.
\end{enumerate}
\end{proof}
\end{proposition}

For reference, here are the tridegrees associated to elementary \emph{cobordism maps}:
\begin{corollary}
Consider the components of the saddle \emph{cobordism map} corresponding to the $d_i$ and $h_i$ maps which define the differential. 
For $d_i$, the $\gr_h$-degree is $i-1$, the $\gr_q$-degree is $2i-3$, and the $\gr_k$-degree is $-2, 0, $ or $2$.
For $h_i$, the $\gr_h$-degree is $i-1$, the $\gr_q$ degree is $2i-1$, and the $\gr_k$ degree is an integer in $[0, i+1]$.
The $(\gr_h, \gr_q, \gr_k)$ triple degree of the \emph{cobordism map} is $(0,1,0)$ for annular births and deaths, and $(0,1,1)$ for non-annular births and deaths.
\end{corollary}

We now state the main properties of the two-dimensional family of annular link invariants $\srt(L,o)$. These features can be compared with corresponding properties of Grigsby-Licata-Wehrli's annular Rasmussen invariants $d_t$, cf. \cite[Theorem 1]{GLW-dt}.

\begin{definition}
The \emph{wrapping number} $\omega(L)$ of an annular link $L$ is the minimal number of (transverse) intersections between the image of $L$ in a diagram $\D(L)$ and an arc connecting $\BBX$ and $\BBO$, over all possible diagrams of $L$.
\end{definition}

\begin{theorem}
\label{thm:main-thm}
Let $(L, o)$ be an oriented annular link. Let $r \in [0,1]$, $t \in [0,1]$.
\begin{enumerate}
    \item \label{thm:annlinkinvt} 
    For each pair $r,t \in [0,1]$, $s_{r,t}(L,o)$ is an oriented annular link invariant.
    \item 
    \label{thm:s00} Suppose $L$ is a knot. Then $s_{0,0}(L,o) =  s_{\F_2}(L,o)-1$, where $s_{\F_2}$ is Rasmussen's concordance invariant over $\F_2$ coefficients \cite{turner-calculating, mackaay-turner-vaz}.
    \item 
    \label{thm:piecewiselinear}
    For fixed $r$ (respectively  $t$), the function $s_{r,t}(L,o)$ is piecewise-linear with respect to the variable $t$ (respectively $r$). 
    \item 
    \label{thm:slopes}
    Let $\omega$ be the wrapping number of $L$. Fix $r \in [0,1)$. Then for all $t_0 \in [0,1)$
    \[
        \left ( -\frac{1}{1-r} \right )
        \lim_{t \to t_0+}
        \frac{
            \srt(L,o) 
            - s_{r,t_0}(L,o)
        }{t - t_0}
        \in \{ -\omega, -\omega + 2, \ldots, \omega-2, \omega\}.
    \]
    %Let $\omega$ be the wrapping number of $L$. Then \[ \left ( - \frac{1}{1-r} \right ) \frac{\partial s_{r,t}(L,o)}{\partial t} \in \{ -\omega, -\omega+2, \ldots, \omega -2, \omega\}\]
    %where $\partial / \partial t$ is the right derivative.
    \item 
    \label{thm:cobordism}
    Let $F: (L, o) \to (L', o')$ be an oriented cobordism between two nonempty, oriented links such that each component of $F$ has a boundary component in $L$. Let $a_0$ be the number of annular births or deaths, $a_1$ the number of saddles, and $b_0$ the number of non-annular births or deaths. Then
    \[
        s_{r,t}(L,o) - s_{r,t} (L', o') \leq (r-1)(a_0 - a_1 + b_0(1-t)).
    \]
    If furthermore each component in $F$ has a boundary component in $L'$ as well, then 
    \[
        \vert s_{r,t}(L,o) - s_{r,t} (L', o') \vert \leq (r-1)(a_0 - a_1 + b_0(1-t)).
    \]
    \label{cobordismproperty}
    \item $\srt$ is an annular concordance invariant.
\end{enumerate}
\begin{proof}
    \begin{enumerate}
        \item Let $\D$ and $\D'$ be two annular link diagrams related by an annular Reidemeister move: that is, an isotopy or Reidemeister move that never crosses either of the marked points $\BBX$ or $\BBO$. Forgetting the location of the basepoints, \SSS\ show in Section 4 of \cite{SSS} that $\CC_\tot(\D) \simeq \CC_\tot(\D')$, i.e.\ are equal in the homotopy category of (bigraded) chain complexes $\CK(\F_2[H,W])$ over $\F_2[H,W]$. Thus they are also equal after applying the functor $\CK(\F_2[H,W]) \to \catFilt$ which sets $H = 1$ and $W= 1$. Hence the $\gr_h$ and $\gr_q$ degrees of the chain homotopy equivalence between $\Cftot(\D)$ and $\Cftot(\D')$ are both 0. 
        
        As Reidemeister moves are local, we can ensure that the small neighborhood where $\D$ differs from $\D'$ is disjoint from the basepoints.
        
        It remains to check that the chain homotopy equivalence between the images $\CC_\ftot(\D)$ and $\CC_\ftot(\D')$ in $\catFilt$ (where $H,W = 1$ throughout) has $\gr_{0,1} = \gr_q - \gr_k$ filtered degree 0. 
        
        The chain homotopy equivalences are defined using \emph{cancellation} (see \cite{SSS} Lemma 4.5).
        Cancellation of a \emph{cancellation data} is a filtered chain homotopy equivalence if the component of the differential in the cancellation data is minimal in $\gr_{0,1}$ degree. 
        
        For Reidemeister I and II, it is easy to see that each cancellation step involves cancelling a differential between resolutions where all circles involved keep their trivial or nontrivial status, or are new trivial circles. 
        Reidemeister III invariance is proven by reducing each complex to a smaller complex, and then noting that the smaller complexes are equivalent, by isotopy. So, it suffices to check that either of these reductions constitutes a filtered chain homotopy equivalence. Again, the cancellation data involve differentials that maintain the annular status of all circles. 
        
        Since the $\gr_q$ degree of these cancellation data is known to be 0, and by Lemma \ref{lem:tridegree}, all components of the differential have $\gr_{0,1}$ degree at least 0. Thus the components of the differential chosen for cancellation are indeed components with minimal $\gr_{0,1}$ degree.

        \item One can prove this using Lemma 5.5 in \cite{SSS}, but perhaps the quickest proof relies on \SSS's invariants. When $r = 0$, $\mathcal{U}_{(0)}[n]$ and $\CF^{0,0}_n$ agree on the lattice $\Z \times (2 \Z + 1)$. Note that the quantum grading of any generator in the Khovanov complex for a knot is an odd integer. Now \SSS's $s^{\mathcal{U}_{(0)}}_o(K)$, which they show is equal to $s_{\F_2}$, is the first \emph{even} integer for which the filtration level $\CF_{\mathcal{U}_{(0)}}[n]$ \emph{does not} contain a representative of $[g(o)]$. Hence $s_{0,0}(K,o) = \gr_q[g(o)]$ must be the odd integer $s^{\mathcal{U}_{(0)}}_o(K)-1 = s_{\F_2}-1$. 
        \item For fixed $r$ (resp.\ $t$), the function $\srt(K,o)$ is piecewise because the complex $\CC_\ftot$ is finitely generated. Along intervals on which the same representative cycle $x \in [g(o)]$ achieves the maximum grading $\grrt(x) = \grrt[g(o)]$, the values $\gr_h(x), \gr_q(x), \gr_k(x)$, and $r$ (resp.\ $t$) are constant. 
        \item Suppose $[y]$ is a nontrivial homology class. Then $\grrt[y]$ is achieved by some representative cycle $y \in [y]$, i.e.\ $\grrt(y) = \grrt [y]$. Write $y = \sum_i y_i$ where each $y_i$ is a distinguished generator. Then $\grrt(y) = \min_i \{ \grrt(y_i)\}$. So, it suffices to prove the statement for distinguished generators. 
        
        Let $x$ be a distinguished generator. Then 
            \[
                \grrt(x) = r \cdot \gr_h(x) + (1-r) \cdot (\gr_q(x) - t \cdot \gr_k(x))
            \]
        so
            \[
                \lim_{t \to t_0^+}  \frac{\grrt(x) - \gr_{r, t_0}(x)}{t-t_0} = (1-r) \cdot (-\gr_k(x)).
            \]
        Now the state of $x$ has at most $\omega$ $v$-circles, each of which contributes $\pm 1$ (for $v_\pm$, respectively) to $\gr_k(x)$, so $\gr_k(x) \in \{ -\omega, -\omega+2, \ldots, \omega-2, \omega\}$. 
    
    \item Let $\phi_F$ denote the chain map associated to the cobordism $F$, and let $\phi_F^*$ be the induced map on homology. 
    Let $x \in [g(0)]$ such that $\grrt(x) = \grrt[g(0)] = s_{r,t}(L, o)$. 
    By Proposition \ref{prop:cob-filt-degrees}, 
    \[
        \grrt(\phi_F(x)) \geq s_{r,t}(L,o) + (1-r)(a_0 - a_1 + b_0(1-t)).
    \]
    By asserting that every component of $F$ has boundary on $(L, o)$, we assure that $\phi_F^*[g(o)]$ is (a nonzero multiple of) $[g(o')]$. (See BN-2 of Proposition 5.3 in \cite{SSS} and references therein.)
    Hence $\phi_F(x) \in [g(o')]$, so $\grrt(\phi_F(x)) \leq \grrt [g(o')] = \srt(L', o')$. 
    Putting these two inequalities together, we obtain 
    \[ 
        \srt(L', o') \geq s_{r,t}(L,o) + (1-r)(a_0 - a_1 + b_0(1-t))
    \]
    and rearrange to obtain the desired inequality.
    To obtain the statement with absolute value, consider the opposite cobordism $-F: (L', o') \to (L, o)$ as well. 
    \item Now consider the case where there is an annular concordance $F$ between $(L,o)$ and $(L', o')$. In particular, $(L,o)$ and $(L', o')$ each have $l$ components, $F$ a disjoint union of $l$ annuli, each with one boundary component in $L$ and the other in $L'$. Recalling the notation from Section \ref{subsec:annular-cobs}, since $F$ is an \emph{annular} concordance, $F$ is disjoint from $U \times I \subset S^3 \times I$. Hence $b_0 = 0$, and $a_0 = a_1$, so we have $|\srt(L, o) - \srt(L', o')| \leq 0$.
    \end{enumerate}
\end{proof}
\end{theorem}

\begin{remark}
The one-parameter family of invariants $s_{r,0}$ for $r \in [0,1]$ is related to the \SSS\ generalized Rasmussen invariants (see Remark \ref{rmk:relate-sss-invts}). The definition of the one-parameter family of $s_{0,t}$ invariants resembles Grigsby-Licata-Wehrli's annular $d_t$ invariants, which specialize to Rasmussen's $s$-invariant with $\mathbb{Q}$ coefficients (at $t=0$ and $ 2$). 
\end{remark}

\begin{remark}
\label{rem:seeds}
Cotton Seed \cite[Remark 6.1]{LS-refinement} has found examples where $s_{\F_2}(K) \neq s_{\mathbb{Q}}(K)$.  
\end{remark}

\subsection{Tangle closures}

Viewing annular links as tangle closures, we can relate the annular horizontal composition with the non-annular isotopy type of their (vertical) tangle composition by a sequence of 1-handle additions.

Before studying the behavior of $\srt$ under such an operation, we need a few topological definitions. See Figure \ref{fig:tangle-defns} for examples.

\begin{figure}
    \centering
    \labellist
    \pinlabel   $T_1$    at                     11.7    60
    \pinlabel   \rotatebox{180}{$T_1$} at       63      60
    \pinlabel   $\mathbb{X}$    at              154     60
    \pinlabel   $T_1$    at                     187     60
    \pinlabel   $T_1$    at                     243     60
    \pinlabel   $\mathbb{X}$    at              274     60
    \endlabellist
    \includegraphics{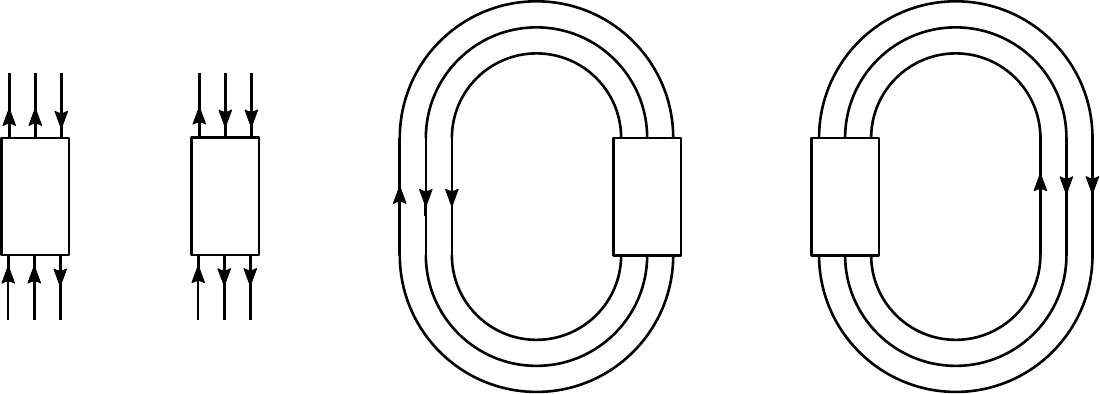}
    \caption{From left to right: the tangles $T_1$ and $\theta T_1$ in $\mathcal{T}^\wedge_{3,3}$ and their closures $\widehat{T_1}$ and $ \widehat{\theta T_1}$, respectively.}
    \label{fig:tangle-defns}
\end{figure}
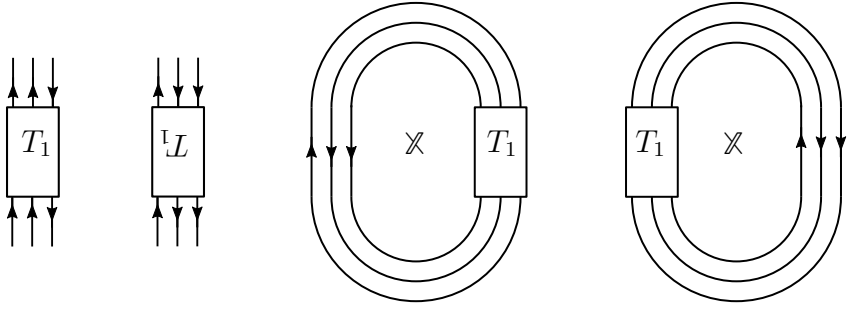

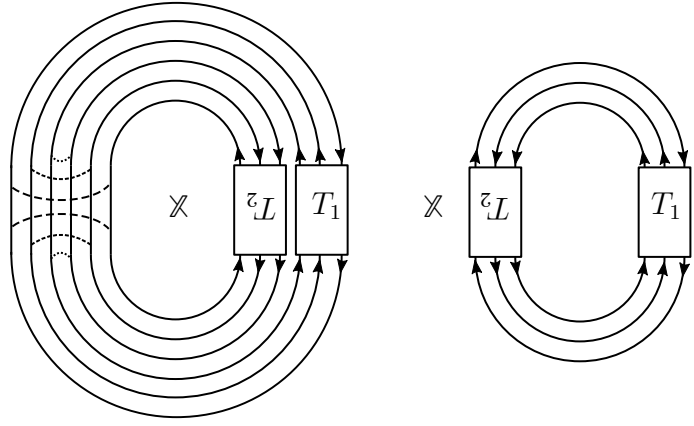
\begin{figure}
    \centering
    \labellist
    \pinlabel	$\mathbb{X}$    at 	64   80
    \pinlabel   $\mathbb{X}$    at  160   80
    \pinlabel   \rotatebox{180}{$T_2$}  at  95 80  
    \pinlabel   $T_1$           at  120 80
    \pinlabel   \rotatebox{180}{$T_2$}  at  183 80  
    \pinlabel   $T_1$           at  248 80
    \endlabellist
    \includegraphics{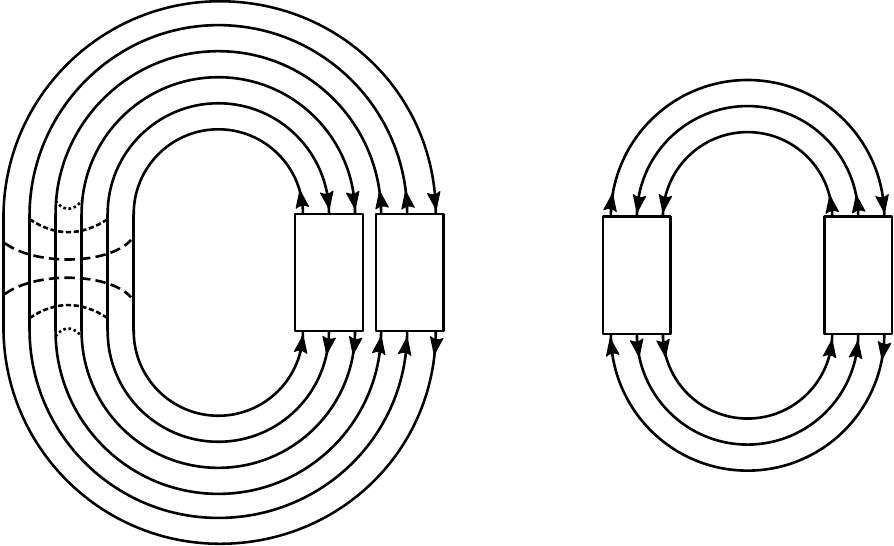}
    \caption{The dotted lines on the left indicate the three saddle cobordisms which transform $\widehat{\theta T_2} \sqcup \widehat{T_1}$ into $\widehat{T_1T_2}$, shown on the right.}
    \label{fig:tangle-compose}
\end{figure}

Let $\nntangles$ be the set of oriented $(n,n)$-tangles whose top and bottom strands agree (i.e.\ the tangle can be closed up). 
Composition of $T_1 \cdot T_2$ is given by stacking, if orientations on the ends permit this. Let $\widehat T$ denote the closure of $T$. If we view $\widehat{T} \subset A \times I$ as the \emph{annular closure} of $T$, there is an ambiguity: we set the convention that the basepoint $\BBX$ is located directly to the left of $T$, and that $\widehat T$ is then closed around $\BBX$, so that in the diagram of $\widehat T$, $T$ is situated on the right. 

Define an involution $\theta$ on annular links which swaps $\widehat T$ with the result of the other convention: diagrammatically, $\theta$ is equivalent to swapping the basepoints $\BBX$ and $\BBO$. 
For $T \in \nntangles$, define $\theta$ as the involution that rotates $T$ by $180^\circ$. Observe that $\theta \widehat{T} = \widehat{\theta T}$.

For an annular link $L$, let $[L]$ denote its (nonannular) isotopy class in $S^3$. Recall (from Remark \ref{rmk:non-annular-comparison}) that if an annular link $L$ is disjoint from the axis (diagrammatically, $L$ is situated in a disk disjoint from an arc connecting $\BBO$ and $\BBX$), then $s_{r,t}(L) = s_{r,0}(L)$ for all $t$. With this in mind, we write $s_r([L]) = s_{r,0}(L)$.

\begin{lemma}
\label{lem:tangle-first-bound}
Let $T_1, T_2 \in \nntangles$. Then 
\[
    \vert 
        s_r([\widehat{T_1 \cdot T_2}]) 
        - \srt(\widehat{T_1} \sqcup \widehat{\theta T_2})
    \vert 
    \leq 
    (1-r)n.
\]
\begin{proof}
By adding $n$ 1-handles to $\widehat{T_1} \sqcup \widehat{\theta T_2}$, we obtain the nonannular representative of $\widehat{T_1 \cdot T_2}$ (see Figure \ref{fig:tangle-compose}). Then inequality then follows from Theorem \ref{thm:main-thm} part \ref{cobordismproperty}. 
\end{proof}
\end{lemma}

\begin{lemma}
\label{lem:tangle-equiv}
Let $T \in \nntangles$. Then for all $r, t \in [0,1]$, 
    $\srt(\widehat T) = \srt(\widehat{\theta T})$.
\begin{proof}
In the hypotheses of Lemma \ref{lem:tangle-first-bound}, let $T_1$ be the identity tangle $\mathbb{1}_n$ and let $T_2 = T$. Combining with Proposition \ref{prop:additive}, we obtain the inequality 
\[
    \vert 
        s_r([\widehat{T_2}]) 
        - \srt(\widehat{T_1}) - \srt(\widehat{\theta T_2})
    \vert 
    \leq 
    (1-r)n.
\]
Now obtain a similar inequality by replacing $T_2$ with $\theta T_2$ and observing that $\widehat{\theta \mathbb{1}_n} = \widehat{\mathbb{1}_n}$ and that $[\widehat{T_2}] = [\widehat{\theta T_2}]$:
\[
    \vert 
        s_r([\widehat{T_2}]) 
        - \srt(\widehat{T_1}) - \srt(\widehat{T_2})
    \vert 
    \leq 
    (1-r)n.
\]
Resolving the absolute values in two different ways and combining these two inequalities forces $\srt(\widehat{\theta T_2}) - \srt(\widehat{T_2}) = 0$.
\end{proof}
\end{lemma}

\begin{theorem}
Let $T_1, T_2 \in \nntangles$ be two composable tangles. Then for $r,t \in [0,1]$,
\[
    \vert
    s_r([\widehat{T_1 \cdot T_2}])
    - \srt(\widehat{T_1} \sqcup \widehat{T_2})
    \vert
    \leq (1-r)n.
\]
In particular, 
\[
    \vert
    s_r([\widehat{T_1 \cdot T_2}])
    - s_r([\widehat{T_1 \cdot \theta T_2}])
    \vert
    \leq 2(1-r)n.
\]
(Note that $[\widehat{T_1 \cdot \theta T_2}] = [\widehat{\theta T_1 \cdot T_2}]$, as $\theta$ acts as identity on links in $S^3$.)
\begin{proof}
The first inequality follows by applying Lemma \ref{lem:tangle-equiv} to the statement of Lemma \ref{lem:tangle-first-bound}. 
The second inequality follows from the first, since both $s_r(\widehat{T_1 \cdot T_2})$ and $s_r(\widehat{T_1 \cdot \theta T_2})$ are related to $\srt(\widehat{T_1} \sqcup \widehat{T_2})= \srt(\widehat{T_1}) + \srt(\widehat{T_2}) = \srt(\widehat{T_1}) + \srt(\widehat{\theta T_2})$. 
\end{proof}
\end{theorem}

\subsection{Applications to braids and more properties}
\label{subsection:braids}
The annular \SSS\ invariants are well-suited to studying annular braid closures equipped with their braid-like orientation. In fact, we will see many similarities in these applications to the annular $d_t$ invariants of Grigsby-Licata-Wehrli \cite{GLW-dt}. 

Let $\sigma  \in \mathfrak{B}_n$ be an $n$-strand braid and $\widehat{\sigma} \subset A \times I$ its annular closure. The braid-like orientation $\sigma_\uparrow$ of $\widehat{\sigma}$ is the one whose strands all wind positively around the braid axis. Its diagram winds counterclockwise about $\BBX$ in $S^2 \setminus \{\BBO, \BBX \}$. We shall abbreviate $s_{r,t}(\widehat{\sigma}, o_\uparrow)$ to $\srt(\sigmahat)$.

\begin{lemma}
Let $\sigma \in \mathfrak{B}_n$ have writhe $w$. Then 
\[(1-r)(w - (1-t)n) \leq \srt(\sigmahat)\]
for all $r\in [0,1]$,  $t \in [0,1]$.
\label{lem:boundsrt}
\end{lemma}
\begin{proof}
We calculate from the definitions:
\begin{align*}
\gr_h(g(o_\uparrow)) = 0,
\quad
\quad
\gr_q(g(o_\uparrow)) = - n + w,
\quad
\quad
\text{and }
\gr_k(g(o_\uparrow)) = - n.
\end{align*}
Thus, for all $r\in [0,1]$,  $t \in [0,1]$, 
\[\grrt(g(o_\uparrow)) = (1-r)(w - (1-t)n) \leq \grrt([g(o_\uparrow)]) = \srt(\sigmahat) \]
as desired.
\end{proof}

When $r=0$, the lower bound on $\srt$ given by Lemma \ref{lem:boundsrt} is analogous to lower bound on the $d_t$ invariant of \cite[Lemma 6]{GLW-dt}. In particular, we recover a  ``$s_{\F_2}$-Bennequin inequality" (proven by Plamenevskaya \cite[Proposition 4]{Plamenevskaya-transverse} and Shumakovitch  \cite[Lemma 4.C]{Shumakovitch-s-Bennequin} originally for $s$ over $\Q$ coefficients).

\begin{corollary}
\label{cor:bennequin}
Let $\sigma \in \mathfrak{B}_n$ have writhe $w$. Then
\[\text{sl}(\sigmahat) \leq s_{\F_2}(\sigmahat) - 1\]
where $s_{\F_2}$ is the Rasmussen concordance invariant over $\F_2$ and $\text{sl}(\sigmahat) = -n + w$ is the self-linking number of the transverse link represented by $\sigmahat$. 
\end{corollary}
\begin{proof}
We specialize the result of Lemma \ref{lem:boundsrt} to $r = 0$ and $t=0$.
\end{proof}
By work of Cotton Seed (see Remark \ref{rem:seeds}), in some cases this produces a stronger upper bound on the self-linking number than the ``$s$-Bennequin inequality" for $s$ over $\Q$. Moreover, Lemma \ref{lem:boundsrt} gives rise to a new family of Bennequin-type inequalities. 

\begin{corollary}
Let $\sigma \in \mathfrak{B}_n$ have writhe $w$. Then 
\[(1-r)\text{sl}(\sigmahat) \leq s_{r,0}(\sigmahat)\]
for all $r \in [0,1]$.
\end{corollary}
\begin{proof}
We specialize the result of Lemma \ref{lem:boundsrt} to $t=0$.
\end{proof}
Furthermore, Theorem \ref{thm:qp} shows that for quasipositive braid closures the $s_{\F_2}$-Bennequin bound is sharp. 

\begin{theorem} 
If $\sigma$ is a quasipositive braid of index $n$ and writhe $w \geq 0$, we have
\[\srt(\sigmahat) = (1-r)(w - (1-t)n) \]
for all $r \in [0,1]$ and $t \in [0,1]$.
\label{thm:qp}
\end{theorem}
\begin{proof}
Lemma \ref{lem:boundsrt} gives us the lower bound 
\[(1-r)(w - (1-t)n)   \leq \srt(\sigmahat).\]
To obtain the upper bound, note that if $\sigmahat$ is quasipositive, there is an oriented (annular) cobordism $F$ from $\sigmahat$ to $\idclosure_n$ obtained by performing an orientable saddle cobordism near each quasipositive generator of $\sigma$ as in \cite[Theorem 2]{GLW-dt} and \cite[Fig.\ 7]{Plamenevskaya-transverse}, and each component of this cobordism has a boundary on $\sigmahat$. Using the crossingless diagram for $\idclosure_n$, we have 
\[\srt(\idclosure_n) = -(1-r)(1-t)n.\]
Part (\ref{cobordismproperty}) of Theorem \ref{thm:main-thm} tells us that $\srt(\sigmahat) - \srt(\idclosure_n) \leq (1-r)w.$ Thus, 
\[\srt(\sigmahat) \leq -(1-r)(1-t)n + (1-r)w \]
as desired.
\end{proof}

\begin{theorem} 
\label{thm:writhe}
Suppose $\sigma \in \mathfrak{B}_n$ has writhe $w$. Then $s_{r, 1}(\sigmahat) = (1-r) w$.
\end{theorem}
\begin{proof}
Let $\D$ be a diagram of the annular braid closure $\sigmahat \subset A \times I$ and let $\CC$ denote the graded vector space underlying the \SSS\ complex. The vector space $\CC$ is generated by resolutions of $\D$ whose circles are labeled by $a = v_-+v_+$ or $b = v_-$. The set of $a/b$ markings of resolutions of $\D$ forms a basis for $\CC$. (This is not a $\Z \oplus \Z \oplus \Z$ \emph{filtered} basis for $(C, \partial_{ftot}$)). We will denote the set of these generators by $S$. We will partition $S$ into three subsets:
\begin{enumerate}
    \item $S_1 = \{g(o)\}$
    \item $S_2 = \{x \in (S \setminus S_1) \mid x \text{ is a labeling of the braid-like resolution of } \D \}$
    \item $S_3 = S \setminus (S_1 \cup S_2)$
\end{enumerate}
(Recall that the braid-like resolution of $\D$ is the oriented resolution for the braid-like orientation $o_\uparrow$.) 
Corresponding to this partition of $S$, there is a direct sum decomposition of $\CC$ into subspaces: $\CC = V_1 \oplus V_2 \oplus V_3$, with $V_i = \mathrm{Span}(S_i)$. 

Let 
\begin{align*}
    p &: \CC \to V_1 \oplus V_2
    \\
    q &: V_1 \oplus V_2 \to V_1
\end{align*}
denote the projection maps. Note that $p$ and $q$ satisfy the following properties. 
\begin{enumerate}
    \item With respect to the $\gr_{r, 1}$-grading on $\CC$, $p$ is a grading-preserving map of graded vector spaces. 
    \item $q \circ p $ is a chain map.
    \item $(q \circ p) (g(o)) = g(o)$.
\end{enumerate}
We can now prove the following claims.
\begin{enumerate}
    \item[\textbf{Claim 1}.] If $z \in \CC$ is a cycle satisfying $[z] = [g(o)]$, then $p(z) \neq 0 \in \CC$. 
    \begin{proof}
    Suppose $z = g(o) + \dftot(x)$ for some $x \in \CC$. Hence
    \[(q \circ p) (z) = (q \circ p)(g(o)) + (q\circ p)(\dftot x) = g(o) + \dftot (q\circ p)(x).  \]
    Thus $[(q \circ p)(z)] = [g(o)]$, and since $g(o)$ is nonzero, this shows that $p(z)$ is nonzero. 
    \end{proof}
    \item[\textbf{Claim 2}.] $s_{r,1} (\sigmahat) \leq w.$
    \begin{proof}
    Let $z$ be a representative of $[g(o)]$ for which $\gr_{r, 1}(z) = \gr_{r, 1}[g(o)]$. Then $p(z)$ is nonzero by Claim 1. Since all elements in $V_1 \oplus V_2$ have $\gr_{r,1} = (1-r) w$, we have $\gr_{r,1}(p(z)) = (1-r) w$. Since $p$ is graded as a map of $\gr_{r,1}$-graded vector spaces, this implies 
    \[ s_{r,1}(\sigmahat) = \gr_{r,1}(z) \leq \gr_{r,1}(p(z))= (1-r)w.\]
    \end{proof}
\end{enumerate}
On the other hand, 
    \[s_{r,1}(\sigmahat) =  \gr_{r, 1}[g(o)] \geq  \gr_{r, 1}(g(o)) = (1-r)w,\]
    so $s_{r,1}(\sigmahat) = (1-r) w$ as desired.
\end{proof}
Recall from \cite{GLW-dt} that the obstruction to being quasipositive from $d_t$ is no more sensitive than the one coming from the sharpness of the $s$-Bennequin bound. Similarly, the obstruction to being quasipositive from $s_{0, t}$ is no more sensitive than the one coming from the sharpness of the $s_{\F_2}$-Bennequin bound.
\begin{corollary}
If $\sigma \in \mathfrak{B}_n$ has writhe $w$, then 
\[ s_{0, t}(\sigmahat) = w - (1-t)n \quad \iff \quad \text{sl}(\sigmahat) = s_{\F_2}(\sigmahat) - 1. \]
%\[ \srt(\sigmahat) = (1-r)(w - (1-t)n) \quad \iff \quad \text{sl}(\sigmahat) = s_{\F_2}(\sigmahat) - 1. \]
\end{corollary}
\begin{proof}
Recalling that $\text{sl}(\sigmahat) = -n+w$, the forward implication follows from setting $r =0$ and $t=0$ and applying part \ref{thm:s00} of Theorem \ref{thm:main-thm}. 

For the converse, note the following:
\begin{enumerate}
    \item The function $s_{0, t}$ is piecewise linear (by part \ref{thm:piecewiselinear} of Theorem \ref{thm:main-thm}).
    \item The slope of $s_{0, t}$ with respect to $t$ is bounded above by $n$ (by part \ref{thm:slopes} of Theorem \ref{thm:main-thm}).
    \item By the hypothesis, $s_{0,0} = -n + w$.
    \item We have $s_{0,1} = w$ by Theorem \ref{thm:writhe}. 
\end{enumerate}
These observations immediately imply the reverse implication. 
\end{proof}

Let
    \[
        \mroto(L,o) = 
        \lim_{t \to t_0^+}
        \frac{
            \srt(L,o) 
            - s_{r,t_0}(L,o)
        }{t - t_0}
    \]
We have a \emph{sufficient} condition for a braid conjugacy class to be right-veering. This property is analogous to that of the $d_t$ invariant (cf. \cite{GLW-dt} Theorem 4). The bounds on $m_{r,t}(L,o)$ by part \ref{thm:slopes} of Theorem \ref{thm:main-thm} will be essential in the following theorem. 

\begin{theorem}
Let $\sigma \in \mathfrak{B}_n$. Fix $r = r_0 < 1$. If $s_{r_0, t}(\sigmahat)$ attains maximal slope at some $t = t_0 < \frac{1}{2}$, that is, $\mroto(\sigmahat) = n$ for some $t_0 \in [0, \frac{1}{2})$, then $\sigma$ is right-veering. 

\begin{proof}
The proof relies on Hubbard-Saltz's annular invariant $\kappa$ \cite{Hubbard-Saltz-kappa}, which has the following property:
if $\sigma \in B_n$ is \emph{not} right-veering, then $\kappa(\sigmahat) = 2$. 

First, we review the definition of $\kappa$ and explain how it is related to the \SSS\ complex. Throughout this proof, let $G$ denote the set of distinguished generators of the Khovanov complex $\CC = \langle g \in G \rangle$, over $\F_2$ (which is equal to $\Cftot$ as a vector space). Let $\sigma \in B_n$. A generic cycle can be thought of as a subset of $G$, as we are working over $\F_2$.

% definition of kappa
Recall from Lemma \ref{lem:tridegree} that $(\CC, d_1)$ is $\gr_k$ filtered: let $\CF_c(\CC) = \langle x \in G \st \gr_k(x) \leq c \rangle$; the $\gr_k$ degree of a component of $d_1$ is either $0$ or $-2$. 

Let $\vminus \in G$ denote \emph{Plamenevskaya's cycle}, the distinguished generator at the braidlike resolution with all circles labeled with $v_-$. Hubbard and Saltz define 
\[
    \kappa(\sigmahat) = n + \min\{ c \st [\vminus] = 0 \in H_*(\CF_c(\CC))\}.
\]

% some computations
We record a few computations for future use:

\begin{itemize}
    \item If $a = \gr_q(x)$ and $b = \gr_{q-k}(x)$, then $\gr_k(x) = a-b$ and $\gr_{0,t}(x) = (1-t)a + tb$. 
    \item Note that $\gr_h(\vminus) = 0$. We have $\gr_q(\vminus) = -n +w$, $\gr_k(\vminus) = -n$, so \[\gr_{r,t}(\vminus) = (1-r)(-n(1-t) + w).\]
\end{itemize}
Note that $\vminus$ is the unique generator with the minimum possible $\gr_k$-grading, namely $\gr_k(\vminus) = -n$, so it indeed is a cycle in $(\CC, d_1)$, and it is the unique cycle with $(\grq, \grqk)$ bigrading $(-n+w, w)$.

% if kappa = 2
Suppose, by way of contradiction, that $\kappa(\sigmahat) = 2$ \emph{and} there exists a time $t_0 \in [0,1/2)$ at which $\mroto(\sigmahat) = (1-r_0)n$.

Since $\kappa(\sigmahat) = 2$, there is some chain $\theta \in \CC$ such that $d_1(\theta) = \vminus \in \CF_{-n}$. Since $d_1$ shifts $\grk$ degree by $0$ or $-2$, $\theta \in \CF_{-n+2}$,  but since $\vminus$ is the unique cycle in $\CF_{-n}$, we know that in fact $\grk(\theta) = -n+2$, and so has homogeneous $(\grh, \grq, \grqk)$ trigrading $(-1,-n+w, w-2)$. 

% if m = n
Since $\mroto(\sigmahat) = (1-r_0)n$, there is a cycle $\xi \subset G$ for which $\grroto(\xi) = \sroto(\sigmahat)$ and $\grk(\xi) = -n$. Hence $\vminus \in \xi$, and is one of the generators in $\xi$ whose gradings determine $\srt$ at $(r_0, t_0)$, i.e.\ 
\[
    \min_{x \in \xi} \{ \grroto(x)\} 
    = \grroto(\vminus) =  (1-r_0)(-n(1-t_0) + w).
\]

% 2 cases
At this point, there are a priori two possibilities. Let $\xi' = \xi \setminus \{\vminus\}$.
\begin{enumerate}
    \item There is a generator $x'  \in \xi'$ achieving this minimum, i.e.\ 
    \[ 
        \grroto(x') = r_0(-1) + (1-r_0)(-n+w - t_0\  \grk(x'))
        = (1-r_0)(-n(1-t_0) + w).
    \]
    Then $s_{r_0, t_0}$ can be computed by measuring $x'$ as well. But this forces $\gr_k(x') = -n$, contradicting $x' \notin \CF(-n)$.
   
    \item Otherwise, for all $x' \in \xi'$
    \[\grroto(x') > \grroto(\vminus).\]
    Recall that $\dftot = \sum_{i=1}^\infty d_i + \sum_{i=1}^\infty h_i$. Let $\xi'' = \xi + (\sum_{i=1}^\infty d_i + \sum_{i=1}^\infty h_i) \theta$, so that $[\xi''] = [g(o)] \in H_*(\CC_\ftot)$. Since $\vminus = d_1(\theta)$,
    \[
        \xi'' =\xi' + \left (\sum_{i=2}^\infty d_i + \sum_{i=1}^\infty h_i \right) \theta.
    \]
    From Lemma \ref{lem:tridegree}, one computes that  $\sum_{i=2}^\infty d_i + \sum_{i=1}^\infty h_i$ increases $\grh$ by at least 1, $\grq$ by at least $2$, and $\grqk$  by at least $0$. Since $(\gr_h, \grq, \grqk)(\theta) = (-1, -n + w, w-2)$, 
    \begin{align*}
    \grroto \left(\left (\sum_{i=2}^\infty d_i + \sum_{i=1}^\infty h_i \right) \theta \right) &\geq (1-r_0)( -n(1-t_0) +w  + 2 - 4 t_0),
    \end{align*}
    which is strictly greater than $\grroto(\vminus) = (1-r_0)(-n(1-t_0) + w)$  when $(1-r_0)(2 - 4t_0) > 0$, i.e.\ when $r_0 \in [0,1)$ and $t_0 \in [0,1/2)$, contradicting the assumption that $\gr(\xi) = \grroto[g(o)]$.
\end{enumerate}

\end{proof}
\label{thm:RV}
\end{theorem}

\begin{remark}
The upper bound $t < \frac{1}{2}$ in Theorem \ref{thm:RV} is reminiscent of the upper bound $t < 1$ on the slope of $\Upsilon$ in Theorem 1.0.4 of \cite{He}.
\end{remark}

\begin{proposition}
\label{prop:maxslope}
Let $\sigma \in \mathfrak{B}_n$ have writhe $w$. If $m_{r,t_0}(\sigmahat) = (1-r)n$ for some $t_0 \in [0, 1)$, then $m_{r,t} = (1-r)n$ for all $t \in [t_0, 1)$. 
\begin{proof}
Recall that there is a unique distinguished basis vector, Plamenevskaya's cycle $\mathbf{v_-}$, with $k$-grading $-n$. In view of part \ref{thm:slopes} of Theorem \ref{thm:main-thm}, this implies that
\[ \gr_{r, t_0}([g(o)]) =  \gr_{r, t_0}(\vminus) = (1-r)((-n+w) + nt_0).\]
Recall that $s_{r,1}(\sigmahat) = (1-r)w$ by Theorem \ref{thm:writhe} and \[ m_{r,t} \leq (1-r) n\] 
by part \ref{thm:slopes} of Theorem  \ref{thm:main-thm}. Therefore, the slope $m_{r,t} = (1-r)n$ for all $t \in [0,1)$. 
\end{proof}
\end{proposition}

We will show that the annular $\srt$ invariants are additive under horizontal composition and behave well under orientation reversal. These properties are also shared with the $d_t$ invariants of \cite{GLW-dt}.
\begin{proposition} 
\label{prop:additive}
Let $(L,o), (L', o') \subset A \times I$ and let $(L,o) \sqcup (L', o') \subset A \times I$ denote their annular composition, as in \cite[Figure 2]{GLW-dt}. Then for all $r \in [0,1]$, $t \in [0,1]$, 
\[\srt((L,o)\sqcup(L',o')) = \srt(L,o) + \srt(L', o') \]
\end{proposition}
\begin{proof}
 As $(\Z \oplus \Z \oplus \Z)$-filtered complexes, 
\[ \CC_{tot}( L\sqcup L') = \CC_{tot}(L) \otimes \CC_{tot}(L').\]

Suppose first that the wrapping numbers of at least one of $L$ or $L'$ is even. Then $g(o \sqcup o') = g(o) \otimes g(o')$, and the statement of the proposition follows.

In the case that the wrapping numbers of $L$ and $L'$ are both odd, we have that 
\[g(o \sqcup o') = g(o) \otimes g(-o') = g(-o) \otimes g(o').\]
The two tensor product expressions on the right are achieved by considering the two possible annular diagrams for $L \sqcup L'$. (Begin with annular diagrams for the individual links, and then identify the inner boundary of one annulus with the outer boundary of the other, or vice versa.) Thus, we have
\begin{align*}
\srt((L,o)\sqcup(L',o')) = \srt(L,o) + \srt(L', -o') = \srt(L,-o) + \srt(L', o').
\end{align*}
In particular, when $L' = \idclosure_1$ is the trivial $1$-braid closure, an easy computation for both orientations $o'$ and $-o'$ on $\idclosure_1$ shows that $\srt(\idclosure_1, o') = \srt(\idclosure_1, -o')$. Thus, we have $\srt(L,o) = \srt(L,-o)$ for any link with odd wrapping number. Therefore, 
\[\srt((L,o)\sqcup(L',o')) = \srt(L,o) + \srt(L', o') \]
as desired. 
\end{proof}

\begin{proposition}
\label{prop:orientationreverse}
Let $(L,o) \subset A \times I$ be an oriented annular link. Then for all $r \in [0,1]$, $t \in [0,1]$, 
\[\srt(L,o) = \srt(L,-o)\]
\end{proposition}
\begin{proof}
By Proposition \ref{prop:additive}, we have
\begin{align*}
\srt((L,o)\sqcup(\idclosure_1,o')) = \srt(L,o) + \srt(\idclosure_1, o').
\end{align*}
By considering the annular diagram for $(L,o)\sqcup(\idclosure_1,o')$ where we identify the outer annulus of the diagram for $(L,o)$ with the inner annulus for $(\idclosure_1, o')$, we have 
\begin{align*}
\srt((L,o)\sqcup(\idclosure_1,o')) = \srt(L,-o) + \srt(\idclosure_1, o').
\end{align*}
The statement follows.
\end{proof}

We study the behavior of a braid under Markov stabilizations. 
\begin{proposition}
\label{prop:stab}
Let $\sigma \in \mathfrak{B}_n$ and suppose $\sigma^\pm \in \mathfrak{B}_{n+1}$ is obtained from $\sigma$ by either a positive or negative Markov stabilization. Then for all $r \in [0,1]$, $t \in [0,1]$,
\[ \srt(\sigmahat) - (1-r)t  \leq \srt(\sigmahat^\pm)  \leq  \srt(\sigmahat) + (1-r)t  \]
\end{proposition}
\begin{proof}
Consider the oriented annular cobordism from $\sigmahat^\pm $ to $\sigmahat \sqcup \idclosure_1$ (the horizontal composition of $\sigmahat$ with the trivial $1$-braid closure) with a single odd-index critical point that resolve the extra $\pm$ crossing. The associated chain map on the \SSS\ complex sends $g(\sigmahat^\pm, o_\uparrow)$ to $g(\sigmahat \sqcup \idclosure_1, o_\uparrow)$. 
By Proposition \ref{prop:cob-filt-degrees}, this map is filtered of degree $(1-r)(-1)$. Therefore,
\[ \srt(\sigmahat^\pm) - (1-r) \leq \srt(\sigmahat \sqcup \idclosure_1).\]
Using additivity of $\srt$ under horizontal composition (cf. Proposition \ref{prop:additive}) and Theorem \ref{thm:qp}, we have
\[ \srt(\sigmahat \sqcup \idclosure_1) = \srt(\sigmahat) - (1-r)(1-t) \]
which gives us one of our two desired inequalities:
\[ \srt(\sigmahat^\pm)  \leq  \srt(\sigmahat) + (1-r)t.  \]

For the other inequality, consider the cobordism $F: \sigmahat \to \sigmahat^\pm$ consisting of one nonannular birth and one saddle. Then the associated filtered chain map $\phi_F$ increases $\grrt$ by at least $-t(1-r)$ by Proposition \ref{prop:cob-filt-degrees}. Hence $\srt(\sigmahat) -t(1-r) \leq \srt(\sigmahat^\pm)$.

\end{proof}

\begin{proposition}
Let $\sigma \in \mathfrak{B}_n$ and suppose that $\sigma^+ \in \mathfrak{B}_{n+1}$ is obtained from $\sigma$ by performing a positive stabilization. For $t < 1$ and sufficiently close to $1$, 
\[m_{r,t}(\sigmahat^+) - m_{r,t}(\sigmahat) \geq 1.\]
\end{proposition}
\begin{proof}
Suppose $\sigma$ has writhe $w$, then $\sigma^+$ has writhe $w+1$. We fix $r$ throughout this proof. By Theorem \ref{thm:writhe}, we have that $s_{r,1}(\sigmahat) = (1-r)w$ and $s_{r,1}(\sigmahat^+) = (1-r)(w+1)$. For the purposes of this proof, let $m = m_{r,t}(\sigmahat)$ and $m^+ = m_{r,t}(\sigmahat^+)$ denote the \emph{right}-hand slope of $\srt$ for $t<1$ and sufficiently close to $1$. Since the functions $\srt(\sigmahat)$ and $\srt(\sigmahat^+)$ are piecewise linear with respect to $t$, for $t < 1$ and sufficiently close to 1, we have that \begin{align*}
\srt(\sigmahat) &= m t + (1-r)w - m
\\
\srt(\sigmahat^+) &= m^+ t + (1-r)(w+1) - m^+
\end{align*}
By Proposition \ref{prop:stab}, we have that 
\begin{align*}
\srt(\sigmahat^+) - \srt(\sigmahat) &\leq (1-r)t \\
(m^+ - m)(t-1) &\leq (1-r)(t-1)
\end{align*} 
Thus, $m^+ - m \geq (1-r)$. Since $m^+-m$ are integers of different parity mod $2$ by part \ref{thm:slopes} of Theorem \ref{thm:main-thm}, the lower bound can be improved to $1$. 
\end{proof}

Inspired by the sets introduced in \cite[Remark 15]{GLW-dt}, we study the following set of braids whose $\srt(\sigmahat)$ invariants attain the maximal slope for some $t_0 \in [0, 1)$, and therefore, by Proposition \ref{prop:maxslope} for all $t \in [t_0,1)$. 

\begin{definition}
Let $t_0 \in [0,1)$. Define
\[\Mt' = \{ \text{Braids } \sigma \mid m_{r,t} = (1-r) n \text{ for all } t \in [t_0, 1) \text{ and all } r \in [0,1] \}. \]
\end{definition}

The following two lemmas show that membership in $\Mt'$ is preserved under positive stabilization. 

%In fact, we will show that the set $\Mt'$ is a monoid. However, we emphasize that our monoids $\Mt'$ are defined differently than the monoids $\mathcal{M}_{t_0}$ defined in Section 6 of \cite{GLW-dt}. 
%\begin{lemma}
%$\Mt^n = \Mt \cap \mathfrak{B}_n$ is a monoid for each $n \in \Z_{\geq 0}.$
%\end{lemma}
%\begin{proof}
%\end{proof}

\begin{lemma}
Let $\sigma \in \mathfrak{B}_n$ have writhe $w$, and suppose $\sigma' \in \mathfrak{B}_n$ is obtained from $\sigma$ by inserting a single positive crossing. Then if $\sigma \in \Mt'$ for some $t_0 \in [0,1)$, then $\sigma' \in \Mt'$. 
\end{lemma}
\begin{proof}
Since $\sigma \in \Mt'$, we know that for each $t \in [t_0, 1)$, we have 
\[\srt(\sigmahat) = \grrt(\vminus(\sigmahat)) =  (1-r)((-n+w) + nt).\]
Note that $\sigma'$ has writhe $w+1$. By applying part \ref{thm:cobordism} of Theorem \ref{thm:main-thm} to the annular saddle cobordism $\sigmahat' \to \sigmahat$ that resolves the single extra positive crossing tells us
\[\srt(\sigmahat') - \srt(\sigmahat) \leq (1-r),\]
and hence,
\[\srt(\sigmahat')  \leq (1-r)((-n+w+1) + nt).\]
On the other hand, 
\[ (1-r)((-n+w+1) + nt) = \grrt(g(\sigmahat')) \leq \srt(\sigmahat').\]
Hence $m_{r,t}(\sigmahat') = (1-r)n$ for all $t \in [t_0, 1)$ as desired. 
\end{proof}

\begin{lemma}
Let $t_0 \in [0,1)$. Let $\sigma \in \mathfrak{B}_n$ and $\sigma' \in \mathfrak{B}_{n'}$, and let $\sigma \sqcup \sigma' \in \mathfrak{B}_{n+n'}$ denote their horizontal composition. If $\sigma, \sigma' \in \Mt',$ then $\sigma \sqcup \sigmahat' \in \Mt'.$
\end{lemma}
\begin{proof}
This follows immediately from Proposition \ref{prop:additive}.
\end{proof}

The behavior of the $\srt$ invariant under positive \emph{destabilization} is currently unknown. As described in the introduction, a complete understanding of the behavior of $\srt$ under positive stabilization and destabilization could potentially yield a new transverse invariant. This invariant may or may not be \emph{effective}, or contain more information than the self-linking number. The question of whether Plamenevskaya's transverse link invariant $\psi$ \cite{Plamenevskaya-transverse} is effective remains open. It is shown in \cite[Theorem 1.2]{Baldwin-Plamenevskaya} that for braids representing Khovanov thin knot types, the vanishing of the transverse link invariant $\psi$ only depends on the $s$-invariant and the self-linking number. Recently, Martin \cite{Martin-dt} used $d_t$ to show that for 3-braid closures, the vanishing of Plamenvskaya's transverse link invariant $\psi$ depends only on the $s$-invariant and and the self-linking number. In this direction, it is natural to ask whether the $\srt$ invariant of a braid closure only depends on $s_{\F_2}$ and the self-linking number.

In another direction, we produce a lower bound on the \emph{band rank} $\rk_n$ of a braid \cite{Rudolph-band-rank}. Given $\beta \in \mathfrak{B}_n$, 
    \[
        \rk_n(\beta):= \min \left \{ c \in \Z^{\geq 0} \ \bigg \vert \  \beta = \prod_{j=1}^c \omega_j \sigma_{i_j}^\pm (\omega_j)^{-1} \text{ for some } \omega_j \in \mathfrak{B}_n \right \},
    \]
    where $\sigma_{i_j}$ denotes the elementary Artin generators. Note that band rank $\rk_n$ is a braid conjugacy class invariant. 
    Topologically, band rank $\rk_n(\beta)$ is the minimum number of half-twist bands (running perpendicularly to the strands) needed to construct a Seifert surface for $\bclosure$ from $n$ disks (the obvious Seifert surface for the identity braid closure in $\mathfrak{B}_n$). 

\begin{proposition}
Let $r \neq 1$. 
Given an oriented cobordism $F$ from $(L,o)$ to $(L',o')$ with $a_0$ annular even index critical points, $a_1$ annular odd index critical points, and $b_0$ non-annular even index critical points, we have
\[
 \left | \frac{s_{r,t}(\hat\beta)}{1-r} + n(1-t) \right | \leq rk_n(\beta).
 \]
 
\begin{proof} % This can be greatly shortened!
Recall $\mathbb{1}_n$ denotes the identity braid in $B_n$.
Observe that $s_{r,t}(\idclosure_n)  
= \gr_{r,t}(v_- \otimes v_- \otimes \cdots \otimes v_-)
= (1-r)(-n + tn) = -n(1-r)(1-t).$

Let $F$ be the cobordism from $\bclosure$ to $\idclosure_n$ realizing the band rank of $\bclosure$, i.e.\ $a_1 = \rk_n(\beta)$.  Then by Proposition \ref{prop:cob-filt-degrees},
\[
s_{r,t}(\bclosure) - s_{r,t}(\idclosure_n)
= s_{r,t}(\bclosure) + n(1-r)(1-t)
\leq (1-r) \rk_n(\beta).
\]
Since $-F$ is a cobordism from $\idclosure_n$ to $\bclosure$ with the same $a_1$, we have  $| s_{r,t}(\bclosure) - s_{r,t}(\idclosure_n) | \leq (1-r) \rk_n(\beta)$. Finally, we divide both sides of the inequality by $1-r > 0$.
\end{proof}
\end{proposition}

\section{Functoriality}
\label{sec:functoriality}

\subsection{Strong functoriality}
In \cite{Saltz}, Saltz defines \emph{strong Khovanov-Floer theories} and proves that any \emph{conic} strong Khovanov-Floer theory is \emph{functorial}.

To motivate the definition of a strong Khovanov-Floer theory, we briefly discuss what functoriality should mean, and how we will use it.

An assignment $\CK: \catDiag \to \catFilt$ (see Definition \ref{def:cat-filt}) is a functor if the following conditions are met. 
First of all, $\CK$ must give a well-defined map on objects. So, if $\CK$ describes how to map a link diagram  $\D$ to a filtered chain complex $\CC$, it must map all equivalent link diagrams to the chain homotopy equivalence class of $\CC$ as well. 
Furthermore, oftentimes $\CK$ will be described using some auxiliary information attached to $\D$; we much also check that ultimately, the assignment $\CK(\D)$ did not depend on this auxiliary information.
Next, $\CK$ must also give a well-defined map on morphisms. Again, it must send two equivalent diagrammatic cobordisms to homotopic maps between filtered complexes. This assignment must also not depend on any auxiliary information used to describe the assignment.

The above ensures that the assignment $\CK$ is a functor. In the special case when $\CK$ resembles the Khovanov homology functor, some of the above functorial properties will follow automatically. Such a functor $\CK$ behaves as a categorification of some skein relations, and should satisfy:
\begin{itemize}
    \item $\CK$ agrees with $\Kh$ on the unknot diagram with no crossings. In terms of morphisms, this should have the structure of a Frobenius algebra with multiplication and comultiplication maps.
    \item $\CK$ sends disjoint unions to tensor products.
    \item $\CK$ is built from a cube of resolutions picture. This property is called conicity. 
\end{itemize}
 
These special structural constraints help prove the functoriality of $\CK$. For example, if $\CK$ is \emph{conic} (in Saltz's vocabulary, because it is built with mapping cones), $\CK$ is essentially a lift of a functor from the cube category, and these are very nice diagrams to understand. For example, handleswap invariance follows from the fact that faces commute in many of these cubical complexes. 

Once such a functor $\CK$ is established, we can consider well-defined homology classes in $\CK(L)$ as link invariants, and study their properties under cobordism. In particular, we may study the filtration gradings of some distinguished homology classes as in this paper.

With all this in mind, below is Saltz's definition of a strong Khovanov-Floer theory. If it seems abstract, we recommend cross-referencing with the proof of Proposition \ref{prop:sss-funct} in Section \ref{subsec:sss-functoriality} for the concrete application needed for our results.

\begin{definition}[Definition 3.8, \cite{Saltz}]
A \emph{strong Khovanov-Floer} theory is an assignment
\[
	\CK: \{\text{Link diagrams with auxiliary data}\} \to \{\text{filtered chain complexes}\} 
\]
satisfying the following conditions:

\begin{enumerate}
\item[(S-1)] 
For two collections $A_\alpha$ and $A_\beta$ of auxiliary data associated to a diagram $\D$, there is a chain homotopy equivalence
  \[
		a^\beta_\alpha: \CK(\D, A_\alpha) \to \CK(\D, A_\beta)
	\]
such that the collections $\{\CK(\D, A_\alpha)\}$ together with the maps $\{a_\alpha^\beta\}$ form a transitive system in $\catFilt$. Let $\CK(\D)$ denote the inverse limit, or the \emph{canonical representative} for this diagram $\D$. 
\item [(S-2)]
For a crossingless diagram of the unknot $\D$, we have $H_{tot}(\CK(\D)) \cong \Kh(\D)$.
\item [(S-3)] 
For a disjoint union of diagrams $\D \cup \D'$, we have a chain homotopy equivalence $\CK(\D \cup \D') \simeq \CK(\D) \otimes_{\F} \CK(\D')$.
\item [(S-4)] If $\D'$ is the result of a diagrammatic handle attachment to $\D$, 
    \begin{itemize}
        \item there is a function $\phi: \Aux(\D) \to \Aux(\D')$,
        \item there is a map $h_{A_\alpha, \phi(A_\alpha), B}: \CK(\D, A_\alpha) \to \CK(\D', \phi(A_\alpha))$ where $B$ is additional auxiliary data for the handle attachment,
        \item for fixed $B$, these maps extend to maps on transitive systems, therefore defining a map $h_B: \CK(\D) \to \CK(\D')$ on canonical representatives, 
        \item and for any two sets of additional auxiliary data $B$ and $B'$, we have $h_B \simeq h_{B'}$.
    \end{itemize}
\item [(S-5)] For $U$ a crossingless diagram of an unknot, $\CK(U)$ is a Frobenius algebra $\CA$ with the operations $\iota: \F \to \CA, \epsilon: \CA \to \F, m: \CA \otimes \CA \to \CA, $ and $\Delta: \CA \to \CA \otimes \CA$ given by diagrammatic birth, death, merge, and split, respectively.
\item [(S-6)] If $\D'$ is the result of a planar isotopy on $\D$, then $\CK(\D) \simeq \CK(\D')$.
\item [(S-7)] If $\D = \D_0 \sqcup \D_1$, $\D' = \D'_0 \sqcup \D'_1$, and $\Sigma$ is a diagrammatic cobordism $\D \to \D'$ that is a disjoint union of cobordisms $\Sigma_0: \D_0 \to \D'_0$ and $\Sigma_1: \D_1 \to \D'_1$, then $\CK(\Sigma) \simeq \CK(\Sigma_0) \otimes \CK(\Sigma_1)$.
\item [(S-8)] The handle attachment maps satisfy handleswap invariance and movie move 15, up to homotopy.
\end{enumerate}

\end{definition}

\begin{definition}[Definition 3.9, \cite{Saltz}]
Let $c$ be some crossing in a link diagram $\D$. For $i = 0,1$, let $\D_i$ be the diagram where $c$ is replaced by its $i$-resolution. Let $\gamma$ be the arc in $\D_0$ along which a 1-handle is attached to produce the diagram $\D_1$. 
A strong Khovanov-Floer theory $\CK$ is \emph{conic} if 
$\CK(\D) \simeq \mathrm{cone}(h_\gamma: \CK(\D_0) \to \CK(\D_1))$
where $h_\gamma$ is the handle attachment map.
\end{definition}

\begin{theorem}[\cite{Saltz}]
Conic, strong Khovanov-Floer theories are functorial. In other words, $\CK$ gives rise to a functor $\catDiag \to \Kom(\Vect)$, the homotopy category of complexes over $\F_2$.
\end{theorem}

\begin{remark}
Saltz dubbed this functoriality property ``strong" to differentiate from the functoriality of Khovanov-Floer spectral sequences described by Baldwin-Hedden-Lobb in \cite{BHL-funct}. The key difference is that even though under some finiteness conditions $E^\infty$ is isomorphic to the total homology of a filtered complex, the map induced between $E^\infty$ pages does not always agree with the map induced on homology.
\end{remark}

%\begin{remark}
%We use the phrase ``chain homotopy equivalent" to indicate isomorphism of objects in the homotopy category of chain complexes.
%\end{remark}

% PERHAPS this remark isn't too relevant; or we can ask it in a questions section
\iffalse
\begin{remark}
We are working over $\F_2$, so we are not affected by the sign ambiguity. Blanchet's work might extend Adam's work. Also see CMW and Caprau perhaps? (These are $Z[i]$.) 
See [Putyra et al] for fixing this in a more general fashion.
% This is all stuff Putyra has told me in his KiW talks.
\end{remark}
\fi

\subsection{Functoriality of the \SSS\ complex}
\label{subsec:sss-functoriality}

In \cite[Theorem 6.9]{Saltz}, Saltz proves that Szab\'o's geometric spectral sequence is a functorial link invariant. Here we prove that the \SSS\ complex is a functorial link invariant. 

\begin{proposition}
\label{prop:sss-funct}
The \SSS\ theory $\CK(\D) = (\CC_{ftot}, \partial_{ftot})$ is a conic, strong Khovanov-Floer theory, and is therefore functorial. 

\begin{proof}

First of all, conicity follows by \SSS's definition of saddle maps \cite[Definition 5.2]{SSS}.

\begin{enumerate}
    \item [(S-1)] 
    %Here we use the first part of (S-4), which is proven by \SSS. See Figure \ref{fig:change-decor}.
    
    \iffalse
    \begin{figure}
        \centering
        \includegraphics[width=4in]{change-decor.PNG}
        \caption{Decoration change. see Figure 4.5 of \SSS}
        \label{fig:change-decor}
    \end{figure}
    \fi
	
	\setlength\tempfigdim{0.08\textheight}
      \begin{figure}
    \[
    \xymatrix@C=0.1\textwidth{
      \vcenter{
      \hbox{\includegraphics[height=\tempfigdim]{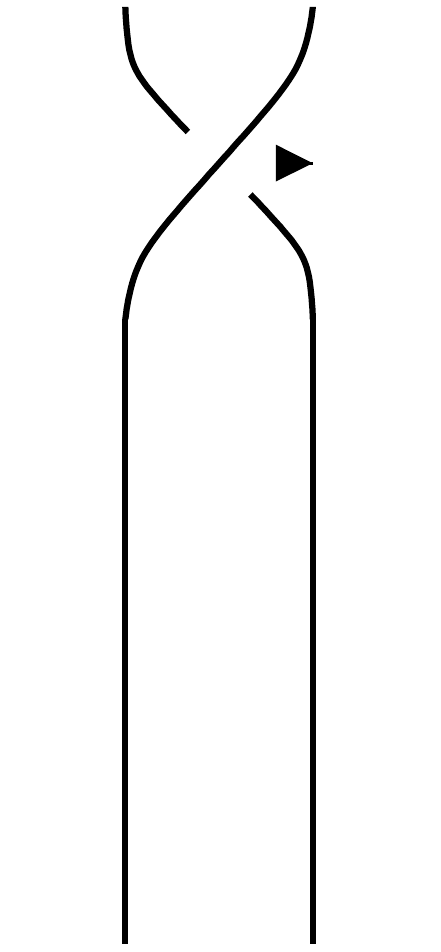}}
      \txt{($\D,  A_-$)} 
      } \ar@{->}[r] &
      \vcenter{\hbox{
      \includegraphics[height=\tempfigdim]{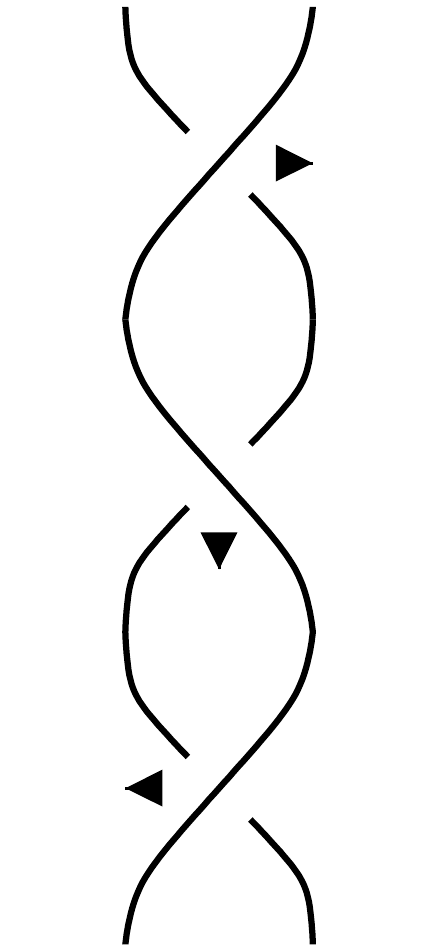}}
      \txt{($\tilde\D,  A_\pm$)} 
      }\ar@{->}[r] &
      \vcenter{\hbox{\includegraphics[height=\tempfigdim]{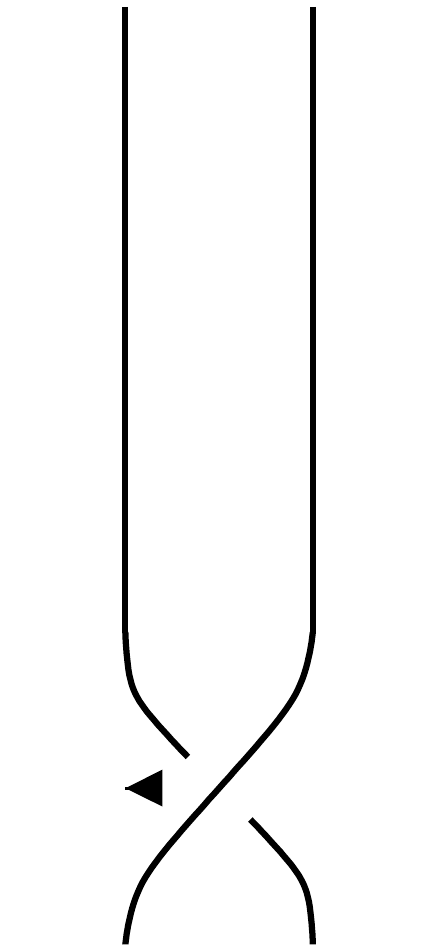}}
      \txt{($\D,  A_+$)}}
    }
    \]
    \caption{Changing the decoration at a single crossing.}
     \label{fig:change-decor}
  \end{figure}
    
    Suppose for a diagram $\D$ two sets of decorations $A_{-}$ and $A_{+}$ differ at exactly one crossing $c$; here we think of $-$ and $+$ as the two ways to decorate the crossing $c$. We form a diagram $\tilde\D$ differing from $\D$ in a small disk around $c$ that replaces $c$ with three crossings, by way of a Reidemeister II move, as in Figure~\ref{fig:change-decor}. We associate a decoration $A_{\pm}$ to $\tilde\D$ as follows: the two outer crossings are labeled with opposite decoration choices, labeled $-$ and $+$, and we choose an arbitrary decoration for the middle crossing. By \cite[Corollary 4.3]{SSS}, the complex $\CK(\tilde \D, A_{\pm})$ is chain homotopy equivalent to each of $\CK( \D, A_{-})$ and $\CK(\D, A_{+})$; \SSS\ use the cancellation lemma to show that these chain isomorphisms are cancellations of acyclic \emph{direct} summands. (A change of basis allows the complex to split.) The chain homotopy equivalence $a_-^+:  \CK( \D, A_{-}) \to \CK(\D, A_{+})$ is the composition
    \[a_-^+ = \widetilde{a}_\pm^+ \circ \widetilde{b}_-^\pm, \]
where $\widetilde{a}_\pm^-: \CK(\tilde \D, A_\pm) \to \CK( \D, A_-) $ and $\widetilde{a}_\pm^+: \CK(\tilde \D, A_\pm) \to \CK(\D, A_+)$ are the \SSS\ Reidemeister II chain homotopy equivalences described in Proposition 4.7 of \cite{SSS}, and $\widetilde{b}_-^\pm : \CK( \D, A_-)  \to \CK(\tilde \D, A_\pm) $ is the chain map induced by the identification of $( \D, A_-)$ with the $``01"$ partial resolution of $(\tilde \D, A_\pm)$ and satisfies $\widetilde{a}_\pm^- \circ \widetilde{b}_-^\pm  \simeq \text{id}$.

    Now for decorations $A_\alpha$ and $A_\beta$ on $\D$ differing at multiple crossings, define the map relating their complexes 
    \[ a_\alpha^\beta : \CK(\D, A_\alpha) \to \CK(\D, A_\beta)\]
    as a composition of the chain homotopy equivalences of single crossing changes. If this is well-defined (i.e.\ does not depend on the order of crossings we choose to compose the chain homotopy equivalences), then these maps form a transitive system.

Thus, without loss of generality, we may assume that the decorations $A_\alpha = A_{-, -}$ and $A_\beta = A_{+, +}$ differ at exactly two crossings, $c_1$ and $c_2$.  Let $A_{\eta'} = A_{+, -}$ and $A_\eta = A_{-, +}$ be the two decorations for $\D$ that differ from $A_\alpha$ by exactly one decoration change, at either $c_1$ or $c_2$, respectively. 
To show that the system of crossing change maps is transitive, we need to check that 
\[ a_{\eta'}^\beta \circ a_\alpha^{\eta'}  \simeq a_\eta^\beta \circ a_\alpha^\eta. \]

By decoration invariance of the \SSS\ complex, $\CK(\tilde{\tilde \D}, A_{\pm, \pm})$ is chain homotopy equivalent to each chain complex $\CK(\D, A_\alpha)$, $\CK(\D, A_\beta)$, $\CK(\D, A_\eta)$, and $\CK(\D, A_{\eta'})$. By symmetry, it suffices to show that the box below commutes up to homotopy (the diagram names are omitted but should be clear from the decoration notation):
    \begin{center}
    \begin{tikzcd}
    \CK(A_{\pm, \pm}) \arrow{r}{\widetilde{\widetilde{a_2}}_\pm^- } \arrow{d}{\widetilde{\widetilde{a_1}}_\pm^-} 
        & \CK(A_{\pm, -}) \arrow{d}{\widetilde{a_1}_\pm^-} \\
    \CK(A_{-, \pm}) \arrow{r}{\widetilde{a_2}_\pm^- } 
        & \CK(A_{-,-}) 
    \end{tikzcd}
    \end{center}
where the chain isomorphism subscripts, $1$ or $2$, indicate the crossing, $c_1$ or $c_2$, associated to the Reidemeister II chain map. A detailed description of each composition requires drawing the 4-dimensional cube of $2^4$ vertices in the partial cube of resolutions corresponding to the four crossings we wish to eliminate, and performing cancellations according to the usual proof (described in Proposition 4.7 of \cite{SSS}), twice, as illustrated in Figure \ref{fig:Reid2-cube}. Ultimately, we find that both chain homotopy equivalences identify $\CK(A_{-, -})$ with the complex $\CK(A_{\pm, \pm})_{(1,0,1,0)}$, situated at the vertex $(1,0,1,0)$ of this 4-dimensional cube.

\setlength\tempfigdim{0.05\textheight}
\begin{figure}
  \[
  \xymatrix{
  \vcenter{\hbox{\includegraphics[height=\tempfigdim]{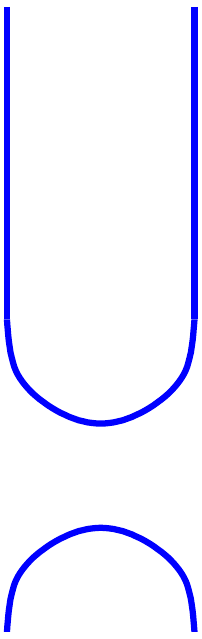}}}
  \quad
    \vcenter{\hbox{\includegraphics[height=\tempfigdim]{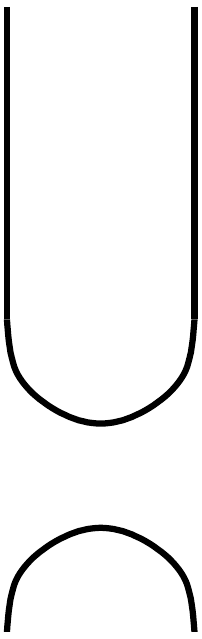}}} \ar[r]\ar[d]
    &
    \vcenter{\hbox{\includegraphics[height=\tempfigdim]{reider2-00_blue}}}
     \quad
    \vcenter{\hbox{\includegraphics[height=\tempfigdim]{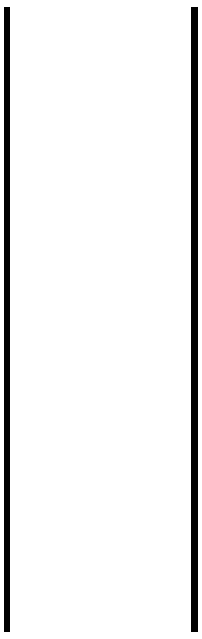}}} \ar[d]
    &&&
    \vcenter{\hbox{\includegraphics[height=\tempfigdim]{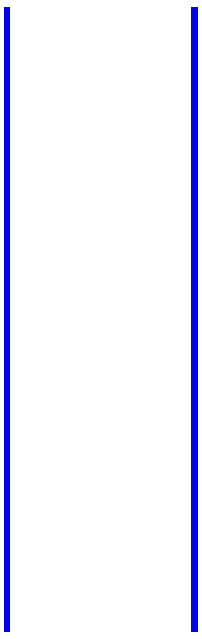}}}
    \quad
    \vcenter{\hbox{\includegraphics[height=\tempfigdim]{reider2-00}}}\ar[r]\ar@{=>}@[violet][d]
    &
    \vcenter{\hbox{\includegraphics[height=\tempfigdim]{reider2-01_blue}}}
    \quad
    \vcenter{\hbox{\includegraphics[height=\tempfigdim]{reider2-01}}}\ar[d]
    \\ % end first row
    \vcenter{\hbox{\includegraphics[height=\tempfigdim]{reider2-00_blue}}}
    \quad
    \vcenter{\hbox{\includegraphics[height=\tempfigdim]{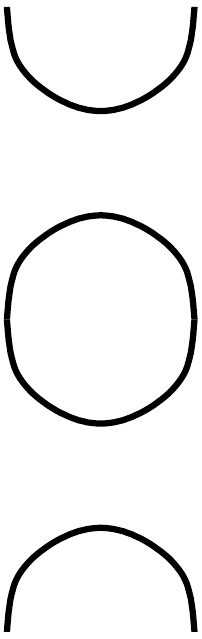}}}\ar[r]
    &
    \vcenter{\hbox{\includegraphics[height=\tempfigdim]{reider2-00_blue}}}
    \quad
    \vcenter{\hbox{\includegraphics[height=\tempfigdim]{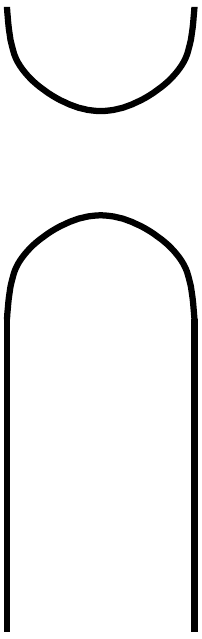}}}
    & \ar@[orange][r] & &
    \vcenter{\hbox{\includegraphics[height=\tempfigdim]{reider2-01_blue}}}
    \quad
    \vcenter{\hbox{\includegraphics[height=\tempfigdim]{reider2-10}}}\ar@{=>}@[violet][r]
    & 
    \vcenter{\hbox{\includegraphics[height=\tempfigdim]{reider2-01_blue}}}
    \quad
    \vcenter{\hbox{\includegraphics[height=\tempfigdim]{reider2-11}}}
    \\ %end second row
    & \ar@{=>}@[orange][d] &  & &  \ar@[orange][d]  &
    \\
    & & & & &
    \\ %end second row
        \vcenter{\hbox{\includegraphics[height=\tempfigdim]{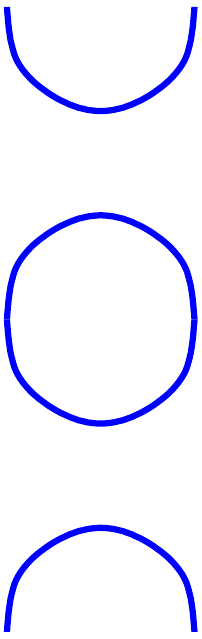}}}
        \quad
        \vcenter{\hbox{\includegraphics[height=\tempfigdim]{reider2-00}}}\ar[r]\ar[d]
    &
    \vcenter{\hbox{\includegraphics[height=\tempfigdim]{reider2-10_blue}}}
        \quad
    \vcenter{\hbox{\includegraphics[height=\tempfigdim]{reider2-01}}}\ar[d]
    &\ar@{=>}@[orange][r] &&
    \vcenter{\hbox{\includegraphics[height=\tempfigdim]{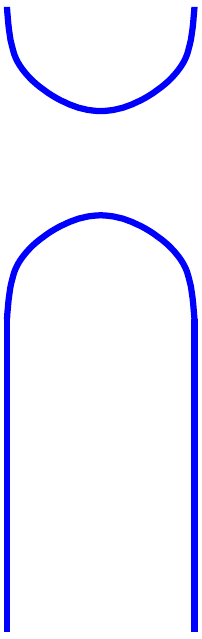}}}
    \quad
    \vcenter{\hbox{\includegraphics[height=\tempfigdim]{reider2-00}}}\ar[r]\ar[d]
    &
    \vcenter{\hbox{\includegraphics[height=\tempfigdim]{reider2-11_blue}}}
    \quad
    \vcenter{\hbox{\includegraphics[height=\tempfigdim]{reider2-01}}}\ar[d]
    \\ %end third row 
    \vcenter{\hbox{\includegraphics[height=\tempfigdim]{reider2-10_blue}}}
        \quad
    \vcenter{\hbox{\includegraphics[height=\tempfigdim]{reider2-10}}}\ar[r]
    &
    \vcenter{\hbox{\includegraphics[height=\tempfigdim]{reider2-10_blue}}}
        \quad
    \vcenter{\hbox{\includegraphics[height=\tempfigdim]{reider2-11}}}
    &&&
    \vcenter{\hbox{\includegraphics[height=\tempfigdim]{reider2-11_blue}}}
    \quad
    \vcenter{\hbox{\includegraphics[height=\tempfigdim]{reider2-10}}} \ar[r]
    &
    \vcenter{\hbox{\includegraphics[height=\tempfigdim]{reider2-11_blue}}}
    \quad
    \vcenter{\hbox{\includegraphics[height=\tempfigdim]{reider2-11}}}
  }
  \]
  \caption{The partial cube of resolutions of $\widetilde{\widetilde{\D}}$ for the RII
    invariance. We first cancel along the orange double arrows, then along the violet double arrows between the $``01**"$ resolutions, leaving a subquotient
    complex isomorphic to the complex for $D$.}
    \label{fig:Reid2-cube}
\end{figure}

    \item[(S-2)] This is clear by definition at the chain level. Since there are no differentials, this also holds in homology.
    
    \item[(S-3)] This holds at the chain level by definition. Furthermore, there are no cross-differentials (i.e.\ differentials are of the form $\phi \otimes \psi$) between the two tensor factors, so this also holds for homology. 

    \item[(S-4)] Let $\D$ and $\D'$ differ by a diagrammatic handle-attachement, and let $\Aux(\D)$ denote the set of decorations for the diagram $\D$. There is a canonical bijection $\phi: Aux(\D) \to Aux(\D')$, in the sense that $\D$ and $\D'$ differ only on a disk on which there are no crossings to decorate. For a 0- or 2-handle attachment, there is no additional auxiliary data $B$ for the cobordism. For a 1-handle attachment, the auxiliary data $B$ is consists of the oriented arc $\gamma$ along which surgery is performed. Let $-\gamma$ denote the oppositely oriented arc.
    
    Let $A_\alpha$ and $A_\beta$ denote decorations for a diagram $\D$. Let $\D_\gamma$ be the diagram obtained from $\D$ by adding an extra crossing as follows: the directed arc $\gamma$ is replaced with a crossing whose 0-resolution is $\D_{\gamma;0}$ and whose 1-resolution is $\D_{\gamma;1}$, decorated to agree with the orientation of $\gamma$. Define $\D_{-\gamma}$ similarly; the only difference is the decoration at the new crossing. Thus, $\D = \D_{\gamma; 0}$ and $\D' = \D_{\gamma;1}$ differ by a diagrammatic one-handle attachment. Associated to the additional auxiliary data of the oriented arc $\gamma$, there are chain homotopy equivalences $a_{\alpha;0}^\beta$ and $a_{\alpha;1}^\beta$ as defined in (S-1), and there exist chain maps $f_{\alpha, \gamma}$ and $f_{\beta, \gamma}$ as in the following diagram:
        \begin{center}
    \begin{tikzcd}
    \CK(\D_{\gamma;0}, A_\alpha) \arrow{r}{a_{\alpha;0}^\beta} \arrow[swap]{d}{f_{\alpha, \gamma}} 
        & \CK(\D_{\gamma;0}, A_\beta) \arrow{d}{f_{\beta, \gamma}} \\
    \CK(\D_{\gamma;1}, \phi(A_\alpha)) \arrow{r}{a_{\alpha; 1}^\beta} 
        & \CK(\D_{\gamma;1}, \phi(A_\beta)) \\
    \end{tikzcd}
    \end{center}
    We will show the above diagram commutes up to homotopy. 
    Observe that the left and right columns are the mapping cone complexes $\cone(f_{\alpha, \gamma}) \simeq \CK(\D_{\gamma}, A_\alpha)$ and $\cone(f_{\beta, \gamma}) \simeq \CK(\D_{\gamma} , A_\beta)$, respectively. Here we abuse notation by using $A_\alpha$ to denote the decoration on $\D_\gamma$ associated to adding the decoration induced by $\gamma$ on the additional crossing. By decoration invariance of the \SSS\ complex up to chain homotopy, as in (S-1) we have a chain homotopy equivalence $a_\alpha^\beta : \CK(\D_{\gamma}, A_\alpha) \to \CK(\D_{\gamma} , A_\beta) $. 
    
    \begin{claim}
    As a chain homotopy equivalence between mapping cones, the chain map $a_\alpha^\beta : \CK(\D_{\gamma}, A_\alpha) \to \CK(\D_{\gamma} , A_\beta) $ is of the form $
    		\begin{bmatrix}
    			a_{\alpha; 0}^\beta & a_{10} \\
    			a_{01} & a_{\alpha;1}^\beta
   		\end{bmatrix}.$
	\begin{proof}
		Write $a_\alpha^\beta = \begin{bmatrix} a_{00} & a_{10} \\ a_{01} & a_{11} \end{bmatrix}$ and consider the diagonal entries. The decoration changes between $A_\alpha$ and $A_\beta$ occur at only crossings that do not include the crossing associated to $\gamma$. Thus, the (Reidemeister II invariance) cancellation data used to define $a_\alpha^\beta$ are the union of the arrows used to define $a_{\alpha; 0}^\beta$ and $a_{\alpha;1}^\beta$. Any zigzag differentials induced by these cancellations must travel between resolutions where crossing $\gamma$ is 0-resolved to resolutions where crossing $\gamma$ is 1-resolved, and therefore do not affect either of the diagonal entries. 
	\end{proof}
    \end{claim}

    In particular, 
    $a_\alpha^\beta = \begin{bmatrix}
        a_{\alpha;0}^\beta & a_{10} 
         \\
        a_{01} & a_{\alpha;1}^\beta 
    \end{bmatrix}$ 
    is a chain map, where $a_{01}: \CK(\D_{\gamma;0}, A_\alpha) \to \CK(\D_{\gamma;1}, \phi(A_\beta))$. Thus, $a_\alpha^\beta \circ \partial_{\D_{\gamma, \alpha}} + \partial_{\D_{\gamma, \beta}} \circ a_\alpha^\beta =0 $, or 
    \[ 
    \begin{bmatrix}
        a_{\alpha;0}^\beta & a_{10} 
         \\
        a_{01} & a_{\alpha;1}^\beta 
    \end{bmatrix}
    \begin{bmatrix}
        \partial_{\gamma;0}^\alpha & 0 
         \\
        f_{\alpha,\gamma} & \partial_{\gamma;1}^\alpha 
    \end{bmatrix}
    +
    \begin{bmatrix}
        \partial_{\gamma;0}^\beta & 0 
         \\
        f_{\beta,\gamma} & \partial_{\gamma;1}^\beta 
    \end{bmatrix}
    \begin{bmatrix}
        a_{\alpha;0}^\beta & a_{10}  
         \\
        a_{01} & a_{\alpha;1}^\beta 
    \end{bmatrix} 
    = 
    0.
    \]
    The bottom left entry of the resulting matrix shows that the square commutes up to homotopy, via the homotopy $a_{01}$. 

    Then, by Lemma 3.6 of \cite{Saltz}, the maps $f_{A_\alpha, \phi(A_\alpha), \gamma}$ and $f_{A_\beta, \phi(A_\beta), \gamma}$ can be extended to a map $f_\gamma$ on the canonical representative of $\CK(\D)$.
    Hence we may write $f_{\gamma}$ (and  $f_{-\gamma}$) as a map $\CK(\D_{\gamma;0}) \to \CK(\D_{\gamma; 1})$, and it remains to show that $f_\gamma \simeq f_{-\gamma}$.

    We write 
    $H_\gamma = \begin{bmatrix}
        0 & 0 
         \\
        h_{\{ \gamma \}}  & 0 
    \end{bmatrix}$ 
    for the change of decoration isomorphism at the crossing corresponding to $\gamma$ (and $-\gamma$), where $h_{\{ \gamma \}} : \CK(\D_{\gamma;0}) \to \CK(\D_{\gamma; 1})$ is the map in Definition 3.6 of \cite{SSS}.
    We will first show that
    \[ \delta_{\ftot, \gamma} + \delta_{\ftot, -\gamma} = H_\gamma \circ \delta_{\ftot, \gamma} + \delta_{\ftot, \gamma} \circ H_\gamma \]

    Recall $\delta_{\ftot, \gamma} = \ddgamma+ \hgamma$, where $\ddgamma = \sum_i d_{i, \D_\gamma} $ and $\hgamma = \sum_i h_{i, \D_\gamma}$. 
    By the proof of Lemma 3.10 \cite{SSS}, 
    \[ \ddgamma + \dd_{\D_{-\gamma}} = H_\gamma \circ \ddgamma + \ddgamma \circ H_\gamma \]
    By Corollary 3.9 of \cite{SSS}, $H_\gamma$ commutes with $\hgamma$. Moreover, the map $\hgamma$ is independent of the choice of orientation of $\gamma$, and hence, $\hgamma = \h_{\D_{-\gamma}}$. Thus, 
    \[ \delta_{\ftot, \gamma} + \delta_{\ftot, -\gamma} = \ddgamma + \hgamma  +  \dd_{\D_{-\gamma}} + \h_{\D_{-\gamma}}=  H_\gamma \circ \delta_{\ftot, \gamma} + \delta_{\ftot, \gamma} \circ H_\gamma. \]
 This immediately implies that 
    \[ f_{ \gamma} + f_{ -\gamma} = h_{\{\gamma\}} \circ \delta_{\ftot, \gamma;0} + \delta_{\ftot, -\gamma;1} \circ h_{\{\gamma\}}  \]
hence $f_\gamma \simeq f_{-\gamma}$.

    \item[(S-5)] The Frobenius algebra is $\F[x]/(x(x+1))$, since the differential for a diagram with only one crossing is just $d_1 + h_1$ which is the Bar-Natan differential.
    
    \item[(S-6)] This condition follows from the proof that the \SSS\ theory $\CC_{tot}$ (and thus, $\CC_\ftot$) is invariant under Reidemeister moves. See Corollary 4.3 of \cite{SSS}. 
    
    \item[(S-7)] Resolution configurations that don't share any active circles won't give cross differentials between the two tensor factors. In the language of \SSS\ the differentials satisfy the \emph{extension rule} (see \cite[Lemma 3.3]{SSS}. We can view the disjoint union cobordism $\Sigma = \Sigma_0 \sqcup \Sigma_1$ as a composition of cobordisms $\Sigma_0 \sqcup \text{id}_{\D_1})$ and $\text{id}_{\D_0'} \sqcup \Sigma_1)$. By the extension rule, we have 
    \[\CK(\Sigma) \simeq (\text{id}_{\D_0'} \otimes \CK(\Sigma_1)) \circ ( \CK(\Sigma_0) \otimes \text{id}_{\D_1}) = \CK(\Sigma_0) \otimes \CK(\Sigma_1). \]
    
    \item[(S-8)]  
    Movie move 15 invariance follows from conicity and Proposition 6.2 of \cite{Saltz}. One can also check this directly, by comparing the identity map to a birth followed by a merge; dually, compare the identity map to a split followed by a death. (Recall that the birth map sends $x \mapsto x \otimes v_+$.)
    
    Handleswap invariance follows from conicity and Proposition 6.1 of \cite{Saltz}. Both of the cobordism maps $h_{\gamma'} \circ h_{\gamma}$ and $h_{\gamma} \circ h_{\gamma'}$ are given by the sum of all the differentials in the complex $\D_{\gamma \gamma'} = \D_{\gamma' \gamma}$ from the $00$-resolution to the $11$-resolution. 
\end{enumerate}
\end{proof}
\end{proposition}

\iffalse % perhaps not necessary
\begin{section}{Questions}
\label{sec:questions}
Can this theory be extended to $\Z$ coefficients?
\end{section}
\fi

%%%%%%%%%%%%%%%%%%%% 
% Bibliography
%%%%%%%%%%%%%%%%%%%%

\bibliographystyle{alpha}
\bibliography{biblio}

\end{document}